\documentclass[reqno]{amsart}
\usepackage[utf8]{inputenc}
\usepackage[T1]{fontenc}
\usepackage{pgf,pgfarrows,pgfnodes,pgfautomata,pgfheaps,pgfshade,hyperref, amssymb}
\usepackage{amssymb}
\usepackage{enumitem}
\usepackage[english]{babel}
\usepackage[capitalize]{cleveref}
\usepackage{mathtools,tikz}
\usepackage[colorinlistoftodos]{todonotes}
\usepackage{soul}

\usepackage{tikz}

\usepackage{xcolor}

\hypersetup{
    colorlinks,
    linkcolor={blue!30!black},
    citecolor={green!50!black},
    urlcolor={blue!80!black}
}
\usepackage{mathrsfs} 
\usepackage{dsfont}

\newtheorem{theorem}{Theorem}[section]
\newtheorem{proposition}[theorem]{Proposition}

\newtheorem{lemma}[theorem]{Lemma}

\newcounter{thmcounter}

\newtheorem{thmintro}[thmcounter]{Theorem}

\newcounter{introthmcounter}

\newtheorem{Maintheorem}[introthmcounter]{Theorem}
\newtheorem{corollary}[theorem]{Corollary}
\theoremstyle{definition}
\newtheorem{definition}[theorem]{Definition}
\newtheorem*{definition*}{Definition}
\newtheorem{question}[theorem]{Question}
\newtheorem*{question*}{Question}
\newcounter{proofcount}
\AtBeginEnvironment{proof}{\stepcounter{proofcount}}

\makeatletter                  
\@addtoreset{claim}{proofcount}
\makeatother

\theoremstyle{remark}

\newtheorem{example}[theorem]{Example}
\newtheorem{remark}[theorem]{Remark}

\def\R{{\mathbb R}}
\def\Z{{\mathbb Z}}

\def\N{{\mathbb N}}

\def\cE{{\mathcal E}}
\def\cB{{\mathcal B}}
\def\cC{{\mathcal C}}
\def\cA{{\mathcal A}}
\def\cL{{\mathcal L}}

\def\cY{{\mathcal Y}}

\def\cN{{\mathcal N}}
\def\cM{{\mathcal M}}

\def\cR{{\mathcal R}}
\def\cS{{\mathcal S}}
\def\cX{{\mathcal X}}
\def\cW{{\mathcal W}}
\def\ie{{i.e.}}  

\def\sT{{\mathscr T}}
\def\sP{{\mathscr P}}

\title[On partial rigidity of S-adic subshifts ]{On partial rigidity of S-adic subshifts}

\author{Sebasti\'an Donoso}
\address{Departamento de Ingenier\'{\i}a
	Matem\'atica and Centro de Modelamiento Ma\-te\-m\'a\-ti\-co, Universidad de Chile and IRL-CNRS 2807, Beauchef 851, Santiago,
	Chile.} 
\email{sdonoso@dim.uchile.cl}

\author{Alejandro Maass}
\address{Departamento de Ingenier\'{\i}a
	Matem\'atica, Centro de Modelamiento Ma\-te\-m\'a\-ti\-co and Millennium Institute Center for Genome Regulation, Santiago, Chile, Universidad de Chile and IRL-CNRS 2807, Beauchef 851, Santiago,
	Chile.}
\email{amaass@dim.uchile.cl}

\author{Trist\'an Radi\'c}
\address{Departamento de Ingenier\'{\i}a
Matem\'atica, Universidad de Chile, Beauchef 851, Santiago, Chile. Department of Mathematics, Northwestern University, 2033 Sheridan Road
Evanston, IL, United States of America.}
\email{tristan.radic@u.northwestern.edu}

\thanks{All authors were partially funded by Centro de Modelamiento Matemático (CMM) FB210005, BASAL funds for centers of excellence from ANID-Chile. The first author was partially funded by ANID/Fondecyt/1241346. The second author was funded by grant ICN2021-044 from the ANID Millennium Science Initiative. All authors are part of the ECOS-ANID grant C21E04 (ECOS210033)}

\subjclass[2020]{Primary: 37A05; Secondary: 37B10, 37B20}

\keywords{partial rigidity, rigidity, S-adic sequences}

\begin{document}
\date{\today}
\maketitle

\begin{abstract}
 We develop combinatorial tools to study partial rigidity within the class of minimal $\mathcal{S}$-adic subshifts. By leveraging the combinatorial data of well-chosen Kakutani-Rokhlin partitions, we establish a necessary and sufficient condition for partial rigidity. Additionally, we provide an explicit expression to compute the partial rigidity rate and an associated partial rigidity sequence.  As applications, we compute the partial rigidity rate for a variety of constant length substitution subshifts, such as the Thue-Morse subshift, where we determine a partial rigidity rate of 2/3. We also exhibit non-rigid substitution subshifts with partial rigidity rates arbitrarily close to 1  and as a consequence, using products of the aforementioned substitutions, we obtain that any number in $[0, 1]$ is the partial rigidity rate of a system.
\end{abstract}

\section{Introduction}
A measure-preserving system $(X, \cX, \mu, T)$ is partially rigid if there exist a constant $\gamma > 0$ (called a constant of partial rigidity) and an increasing sequence $(n_k)_{k \in \N}$ such that $\displaystyle \liminf_{k \to \infty} \mu (A \cap T^{n_k}A) \geq \gamma \mu(A)$ for every measurable set $A$. When $\gamma=1$, the system is said to be rigid, and the sequence $(n_k)_{k \in \N}$ is called a rigidity sequence. The rigidity notion was introduced by Furstenberg and Weiss in \cite{Furstenberg_Weiss77} and can be regarded as the obstruction to mild mixing (a system is mildly mixing if and only if it does not posses any non-trivial rigid factor). Rigidity, rigidity sequences and their topological counterparts have been extensively studied (refer, for example, to \cite{Bergelson_delJunco_Lemanczyk_Rosenblatt_rigidity_nonrecurrence:2014,Coronel_Maass_Shao_seq_entropy_rigid:2009,Donoso_Shao_uniform_rigid_models:2017,Fayad_Kanigowski_rigidity_wm_rotation:2015,Glasner_Maon_rigidity_topological:1989}).

Friedman is credited with pioneering the concept of partial rigidity in his seminal paper \cite{Friedman_partial_mixing_rigidity_factors:1989} and we use King's definition from \cite{King_joining-rank_finite_mixing:1988}.  In early studies, this notion was closely intertwined with the partial mixing property that is, there exists a sequence $(n_k)_{k \in \N}$ and a constant $\alpha>0$ such that $\displaystyle \liminf_{k \to \infty} \mu(A \cap T^{-n_k}B) \geq \alpha \mu(A)\mu(B)$ for all $A,B \in \cX$ and was primarily explored for rank one systems
(see \cite{Friedman_partial_mixing_rigidity_factors:1989,Goodson_Ryzhikov_conj_joinings_producs_rank1:1997,King_joining-rank_finite_mixing:1988}). Besides that, the exploration of the partial rigid property has been significantly less extensive than that of rigid properties. Little is known about partial rigidity constants or partial rigidity sequences, let alone explicit calculations of the partial rigidity rate $\delta_{\mu}$, that is, the supremum of the partial rigidity constants of a system.
 
Indeed, for other types of systems, the study of partial rigidity has been only considered as a cause of absence of mixing. For instance, the non-mixing property was established for substitution subshifts in 1978 \cite{Dekking_Keane_mixing_substitutions:1978}, followed by interval exchange transformations in 1980 \cite{Katok_interval_exchange_not_mixing:1980}, linearly recurrent subshifts in 2003 \cite{Cortez_Durand_Host_Maass_continuous_measurable_eigen_LR:2003} and exact finite rank Bratteli-Vershik systems in 2013 \cite{Bezuglyi_Kwiatkowski_Medynets_Solomyak_Finite_rank_Bratteli:2013}. In more recent works, the partial rigidity property has been highlighted by the works of Danilenko \cite{Danilenko_finite_rank_rationalerg_partial_rigidity:2016}, who showed that indeed the aforementioned classes of systems are partially rigid (hence they cannot be mixing) and by Creutz \cite{Creutz_mixing_minimal_comp:2023} who proved that non-superlinear complexity subshifts are partially rigid. In the context of interval exchange transformations, we remark that there exists a previous study by Ryzhikov that incorporated all the necessary points to establish partial rigidity \cite{Ryzhikov_absence_mixing:1994}. 

 Despite the number of results presented above for substitution, linearly recurrent and non-superlinear complexity subshifts, for the more general class of $\cS$-adic subshifts there is not a real unified framework allowing one to study partial rigidity and determine the partial rigidity rate of a system in this class. Our aim in this paper is to provide tools based on the combinatorial data inherent in any $\cS$-adic subshift to address these problems. 

The class of $\cS$-adic subshifts of finite alphabet rank (see \cref{subsec:Cantor_defs}) is a natural class of subshifts to study partial rigidity. This class contains well-studied classes of low complexity subshifts, such as substitutions, linearly recurrent subshifts and more generally, non-superlinear complexity subshifts \cite{Donoso_Durand_Maass_Petite_interplay_finite_rank_Sadic:2021}. Although systems within this class may be combinatorially more complex than non super-linear complexity subshits, they share structural dynamical properties such as having zero topological entropy, having a finite number of ergodic measures \cite{Giordano_Putman_Skau_Topological_orbit_equiv_crossed_products:1995}, having a restrictive structure of its automorphisms group \cite{Espinoza_Maass_automorphism_S-adic:2022}, and having a finite number of factors \cite{Espinoza_symbolic_factors_finite_S-adic:2023}. 

In this paper, using the concepts of complete words and Kakutani-Rokhlin towers, we provide a combinatorial insight to the notion of partial rigidity, which is valid beyond the realm of finite rank $\cS$-adic subshifts. 
We first state a very general necessary and sufficient condition for an ergodic and non-atomic measure preserving system to be partially rigid. Although this theorem may seem technically intricate, it can be applied in examples and is the main ingredient in the proofs of other results in the article. To state it, we need the following definitions. 

A nonempty word $w$ is {\em a complete return to a letter}, or just {\em complete} for short, if its first and last letters coincide. For a {\em standard} Kakutani-Rokhlin partition of towers $\sT_1, ..., \sT_d$, and a complete word $w = w_1 \ldots w_{\ell} \in \{1,\ldots,d\}^{*}$ ($w_1=w_{\ell}$),  we let $\sT_w$ denote the set of points that start and end in tower $\sT_{w_1}$ crossing consecutively the towers $\sT_{w_2},\ldots,\sT_{w_{\ell-1}}$ (see \cref{def subtowers} for the precise definition). 

We denote $u \sim w$ if $u = u_1 \ldots u_r$ and $w = w_1 \ldots w_{\ell}$ are complete words and $h_{u_1} + \ldots + h_{u_{r-1}}= h_{w_1} + \ldots + h_{w_{\ell-1}}$, where $h_i$ is the height of the tower $\sT_i$.

\begin{thmintro} \label{theorem KR intro}
Let $(X, \cX, \mu, T)$ be an ergodic system with $\mu$ a non-atomic measure and let $(\sP^{(n)})_{n \in \N}$ be a sequence of Kakutani-Rokhlin partitions satisfying standard assumptions called (KR1)-(KR4). Then, the following properties are equivalent:
\begin{enumerate}[label=\roman*)]
\item  $(X, \cX, \mu, T)$ is $\gamma$-rigid.
\item  There exists a sequence of complete words $(w(n))_{n \in \N}$ such that

\begin{equation} \label{eq teo intro}
 \limsup_{n \to \infty} \sum_{u \sim w(n)} \mu( \sT^{(n)}_{u}) \geq \gamma.
\end{equation}

\end{enumerate}
\end{thmintro}
Here, the superindex $(n)$ in $\sT^{(n)}_{u}$ denotes that we use the towers associated with partition $\sP^{(n)}$. 

The following theorem ensures the existence of a good sequence of complete words. Recall that $\delta_{\mu}$ is the supremum of all partial rigidity constants of the system $(X, \cX, \mu, T)$. 
\begin{thmintro} 
\label{theorem partial rigidity rate intro}
Under the same assumptions of \cref{theorem KR intro}, there exists a sequence of complete words $(w(n))_{n \in \N}$ such that:
\begin{equation} 
\delta_{\mu} = \lim_{n \to \infty} \sum_{u \sim w(n)} \mu( \sT^{(n)}_{u} ) .
\end{equation}
\end{thmintro}

Using this characterization, we derive that $\delta_{\mu}$ is itself a constant of partial rigidity. This has been pointed out in the past (see for instance \cite[Section 1]{King_joining-rank_finite_mixing:1988}), but our proof has the advantage that is constructive and allows us to compute such a constant for different classes of $\cS$-adic subshifts, while giving explicit partial rigidity sequences. Indeed, the proof of \cref{theorem KR intro} shows that, by constructing a sequence of complete words $(w(n))_{n \in \N}$, the sequence $n_k=(h_{w(k)_1}^{(k)} + \ldots + h_{w(k)_{|w(k)|-1}}^{(k)})_{k \in \N}$ is a sequence of partial rigidity.

In the context of $\cS$-adic subshifts, there is a natural sequence of Kakutani-Rohklin partitions where \cref{theorem KR intro} and \cref{theorem partial rigidity rate intro} can be applied directly. From this we derive in \cref{subsec:conditions_partial_rigidity} several sufficient conditions for partial rigidity, specially relevant for subshifts of finite alphabet rank. These include the proportionality of the tower heights or the repetition of a positive morphism in the directive sequence defining the $\cS$-adic subshift (\cref{alturas A mu} and \cref{prop morfismos repetidos}). We remark that these conditions can be tested in concrete examples, as long as we have $\cS$-adic representations of them. Indeed, in \cref{sec m-consecutive}, we introduce a rich family of $\cS$-adic subshifts, which allows us to illustrate our methods and to construct a partially rigid superlinear complexity subshift.

For substitution subshifts, we derive an expression for the partial rigidity rate in terms of measures of cylinder sets. More precisely, we show

\begin{thmintro} \label{thm subtitutions intro}
 Let $\sigma: \cA^* \rightarrow \cA^*$ be a primitive and constant length substitution. Let $\mu$ be the unique invariant measure on the substitution subshift $(X_{\sigma},S)$. Then, 
\begin{equation}
\delta_{\mu} = \sup_{\ell \geq 2} \sum_{\substack{w \text{ complete word} \\ \text{ on } \cL(X_{\sigma}) \text{ and } \lvert w \rvert  = \ell }} \mu (  [w]_{X_{\sigma}} ).
\end{equation}
where $[w]_{X_{\sigma}}$ is the cylinder set given by $w$ on $X_{\sigma}$.
\end{thmintro}
A similar result can be derived for $\cS$-adic subshifts with constant length directive sequence, see \cref{theorem toeplitz delta mu}. 

We obtain as applications that the partial rigidity rate for the Thue-Morse subshift is $2/3$ (\cref{thm:partial_rig_ThueMorse}) and construct for every $\varepsilon >0$ a substitution subshift such that $1 - \varepsilon < \delta_{\mu} <1$ (\cref{cor muy rigido}).  Using these substitutions, we prove in \cref{thm arbitrary partial rigidity rate} that any number in $[0,1]$ is the partial rigidity rate of a measure preserving system. 

\textbf{Organization.} In \cref{sec preliminaries} we provide the essential background in measure preserving, topological and symbolic dynamics needed in this article. 
Section \ref{sec:partial_rigidity} is devoted to stating the necessary and sufficient condition for a system to be partially rigid. The introduction and characterization of the partial rigidity rate are presented in Section \ref{sec:partial_rigid_rate}. 
Sections \ref{section partial rigidity Sadic} and \ref{sec:S-adic rigid} delve into the study of partial rigidity and rigidity within the context of $\cS$-adic subshifts. 
The computation of partial rigidity rates for constant length substitution subshifts is presented in Section \ref{sec p rigid constant}, together with some applications. Except for the proof of \cref{theorem toeplitz delta mu}, this last section can be read independently. We end the document with some open questions.

\textbf{Acknowledgements.}
The authors thank R. Gutiérrez-Romo and S. Petite for helpful discussions. Special thanks to V. Delecroix for insightful comments and for the references presented in \cref{sec:Thue-Morse family}. We also thank the anonymous referee for their comments, which have enriched the article, and for pointing out the existence of Friedman's result, analogous to the one presented in \cref{sec:every number}.

\section{Preliminaries}  \label{sec preliminaries}

\subsection{Measure preserving and topological systems}
A \emph{measure preserving system} (or a \emph{system} for simplicity) is a tuple $(X,\mathcal{X},\mu,T)$, where $(X,\mathcal{X},\mu)$ is a probability space and $T\colon X\to X$ is a measurable and measure preserving transformation. That is, $T^{-1}A\in\mathcal{X}$ and $\mu(T^{-1}A)=\mu(A)$ for all $A\in \cX$. In this paper, we assume without loss of generality that $T$ is invertible, with a measurable and measure preserving inverse $T^{-1}$. We will also assume, without loss of generality, that $(X, \cX, \mu)$ is a standard probability space.

We say that $(X,\cX,\mu,T)$ is \emph{ergodic} (or simply that $T$ or $\mu$ is ergodic) if whenever $A \in \cX$ verifies that $\mu(A\Delta T^{-1}A)=0$, then $\mu(A)=0$ or 1.
A strictly stronger notion is that of mixing. We say that $(X,\cX,\mu,T)$ is \emph{mixing} if $\mu(A\cap T^{-n}B)\to \mu(A)\mu(B)$ as $n$ goes to infinity, for all measurable sets $A,B \in \cX$.

A measure preserving system $(X,\mathcal{X},\mu,T)$ is {\em partially rigid} if there exists an increasing sequence of positive integers $(n_k)_{k \in \N}$ and a constant $\gamma >0$ such that $\displaystyle\liminf_{k\to \infty} \mu(A \cap T^{-n_k}A) \geq \gamma \mu(A)$ for any measurable set $A \in \cX$. In such a case, we also say that the system, or just $\mu$, is $\gamma$-rigid and the sequence $(n_k)_{k \in \N}$ is said to be a {\em partial rigidity sequence}. It is not complicated to prove that to state partial rigidity one only needs to consider sets $A$ in a semi-algebra that generates $\cX$  and that any subsequence of a partial rigidity sequence remains a partial rigidity sequence for the same constant $\gamma >0$.

The case $\gamma=1$ has been extensively studied in the past. In this case, the measure preserving system $(X,\mathcal{X},\mu,T)$ is said to be {\em rigid} and the limit inferior is actually a limit. By standard density arguments, the system is rigid if and only if for any  $f\in L^{2}(\mu)$ we have that $\|f-f\circ T^{n_k}\|_{2}$ goes to 0 as $k$ goes to infinity, for an increasing sequence of positive integers $(n_k)_{k\in\N}$. In this context, the sequence $(n_k)_{k \in \N}$ is said to be a {\em rigidity sequence} for $(X, \cX, \mu,T)$. Note that for every positive integer $c$, the sequence $(c n_k)_{k \in \N}$ is also a rigidity sequence for $(X, \cX, \mu,T)$. 
For recent developments on the rigidity notion and the sequences one can obtain, we refer to \cite{Bergelson_delJunco_Lemanczyk_Rosenblatt_rigidity_nonrecurrence:2014}. 

\begin{remark} \label{remark periodic}
The study of partial rigidity for ergodic systems is only interesting when the system is non-periodic and so we will consider only non-atomic ergodic measures. 
\end{remark}

A measure preserving system is {\em mildly mixing} if for every measurable set $A \in \cX$ with $\mu(A) > 0$, $\displaystyle\inf_{n \in \N} \mu (A \Delta T^{-n}A) >0$. A measure preserving system $(X, \cX, \mu, T)$ is {\em weakly mixing} if the product system $(X\times X,\cX\otimes \cX,\mu\times \mu,T\times T)$ is ergodic. 
It is well known that mixing implies mildly mixing, mildly mixing implies weakly mixing, and weakly mixing implies ergodicity, and all implications are strict. For example, there are weakly mixing and rigid systems (see \cite{Bergelson_delJunco_Lemanczyk_Rosenblatt_rigidity_nonrecurrence:2014}) and also partially rigid and mildly mixing systems (see \cite{Chacon_weakly_mixing_nonstrongly_mixing:1969}).

Let $(X, \cX, \mu,T)$ and $(Y, \cY, \nu, S)$ be two measure preserving systems. If there exists a measurable map $\pi \colon X \to Y$ such that $\pi \circ T = S \circ \pi$ and $\mu(\pi^{-1}A) = \nu (A)$ for all $A \in \cY$, we say that $\pi$ is a {\em measurable factor map}, $(Y, \cY, \nu, S)$ is a {\em measurable factor} of $(X, \cX, \mu,T)$ and $(X, \cX, \mu,T)$ is a {\em measurable extension} of $(Y, \cY, \nu, S)$. When $\pi$ is invertible and $\pi^{-1}$ is a measurable map, $(X, \cX, \mu,T)$ and $(Y, \cY, \nu, S)$ are said to be {\em measurably isomorphic}. We have that ergodicity, partial rigidity, rigidity, and mixing properties are inherited via measurable factor maps. 

A {\em topological dynamical system} is a pair $(X,T)$, where $X$ is a compact metric space and $T$ is a self-homeomorphism. The {\em orbit} of a point $x\in X$ is the set ${\rm orb}(x)=\{T^n x: n\in \Z\}$. We say that $(X,T)$ is {\em minimal} if the orbit of any point is dense in $X$. A point $x \in X$ is {\em periodic} if its orbit is finite and {\em aperiodic} otherwise. 

Let $(X,T)$ and $(Y,S)$ be two topological dynamical systems. If there is an onto continuous map $\pi\colon X \to Y$ such that $\pi \circ T = S \circ \pi $, then we say that $\pi$ is a {\em factor map}, $(Y,S)$ is a {\em factor} of $(X,T)$ and $(X,T)$ is an {\em extension} of $(Y,S)$.  When $\pi$ is a homeomorphism, $(X,T)$ and $(Y,S)$ are said to be {\em topologically conjugate}. Remark that minimality is preserved under factor maps. 

Given a topological dynamical system $(X,T)$, we let $\cM(X,T)$ (resp. $\cE(X,T)$) denote the set of Borel $T$-invariant probability (resp. the set of ergodic probability measures). For any topological dynamical system $\cE(X,T)$ is nonempty and when $\cE(X,T) = \{ \mu\}$ the system is said to be {\em uniquely ergodic}. For every $\nu \in \cM(X,T)$, there is a probability measure $\rho$ in $\cM(X,T)$ such that $\nu(A) = \int_{\cE(X,T)} \mu(A) d \rho(\mu)$ for every Borel measurable set $A$. This is called the ergodic decomposition of $\nu$.

\begin{remark} \label{remark partially rigid invariant measures}
If every ergodic measure $\mu \in \cE(X,T)$ is $\gamma$-rigid for the same partial rigidity sequence $(n_k)_{k \in \N}$, then every invariant measure $\nu \in \cM(X,T)$ is $\gamma$-rigid. Indeed, for any Borel measurable set $A$,

\begin{align*}
\liminf_{k \to \infty} \nu(A \cap & T^{-n_k}A) = \liminf_{k \to \infty} \int_{\cE(X,T)} \mu(A \cap T^{-n_k}A)  d \rho(\mu), \\ 
&\geq  \int_{\cE(X,T)} \liminf_{k \to \infty} \mu(A \cap T^{-n_k}A) \geq \int_{\cE(X,T)} \gamma  \mu(A) d \rho(\mu) = \gamma \nu(A).
\end{align*}

\end{remark}

\subsection{Cantor and symbolic systems.} \label{subsec:Cantor_defs}

A \emph{Cantor system} $(X,T)$ is a topological dynamical system such that $X$ is a \emph{Cantor space}, i.e., the topology of $X$ has a countable basis of \emph{clopen sets} (closed and open sets) and it has no isolated points. Cantor spaces are compact and metric. An important class of Cantor systems is symbolic systems. 

Let $\cA$ be a finite set that we call {\em alphabet}. The elements in $\cA$ are called {\em letters} or {\em symbols}. For $\ell \in \N$, the set of concatenations of $\ell$ letters is denoted by $\cA^{\ell}$ and $w = w_1 \ldots w_{\ell} \in \cA^{\ell}$ is said to be a {\em word} of length $\ell$. The length of a word $w$ is denoted by $|w|$. We set $\cA^* = \bigcup_{n \in \N} \cA^{\ell}$, where, by convention, $\cA^0 = \{ \varepsilon \}$ and $\varepsilon$ is the {\em empty} word. 

Given two words $u$ and $v$, we say that they are, respectively, a {\em prefix} and a {\em suffix} of the word $uv$. For a word $w = w_1  \ldots w_{\ell}$ and two integers $1 \leq i < j \leq \ell$, we write $w_{[i, j+1)} = w_{[i, j]} = w_i  \ldots w_j$. We say that $u$ {\em appears} or {\em occurs} in $w $ if there is an index $ 1 \leq i \leq |w|$ such that $u=w_{[i,i+|u|)}$ and we denote this by $u \sqsubseteq w$. The index $i$ is an {\em occurrence} of $u$ in $w$ and $|w|_u$ is the number of occurrences of $u$ in $w$. 

The set of one-sided sequences $(x_n)_{n \in \N}$ in $\cA$ is denoted by $\cA^{\N}$ and the set of two-sided sequences $(x_n)_{n \in \Z}$ in $\cA$ is denoted by $\cA^{\Z}$. The notation and concepts introduced for words can be naturally extended for sequences in $\cA^{\N}$ and $\cA^{\Z}$. 

The {\em shift map} $S\colon \cA^{\Z} \to \cA^{\Z}$ is defined by $S((x_n)_{n \in \Z})= (x_{n+1})_{n \in \Z}$. A {\em subshift} is a topological dynamical system $(X,S)$, where $X$ is a closed and $S$-invariant subset of $\cA^{\Z}$ endowed with the product topology. Usually, $X$ itself is said to be a subshift. Let $X$ be a subshift on the alphabet $\cA$. Given $x \in X$, the \emph{language} $\cL(x)$ is the set of all words appearing in $x$ and $\cL(X) = \bigcup_{x \in X} \cL(x)$. For two words $u,v \in \cL(X)$ the \emph{cylinder set} $[u \cdot v]_X$ is given by $\{x \in X : x_{[-|u|,|v|)} = uv \}$. When $u$ is the empty word, we only write $[v]_X$. Cylinder sets are clopen sets that form a base for the topology of the subshift. 

A word $w\in \cA^*$ is said to be \emph{complete} if $|w|\geq 2$ and $w_1 = w_{|w|}$. The set of complete words in $\cL(X)$ is denoted by $\cC \cL (X)$. For $u \in \cL(X)$, a \emph{right return word} to $u$ is a nonempty word $w$ such that $uw \in \cL(X)$, $uw$ has $u$ as proper suffix and $uw$ does not have an occurrence of $u$ which is not a prefix or suffix.
Symmetrically, one defines \emph{left return words} to $u$. We denote by $\cR_X(u)$ (resp.  $\cR_X'(u)$) the set of right (resp. left) return words to $u$. Every complete word $v$ starting with a letter $a \in \cA$ satisfies that $v = a w_1 \ldots w_r$, where $w_1,\ldots, w_r \in \cR_X(a)$. If $(X,S)$ is minimal, then $|\cR_X(u)| < \infty $ for every $u \in \cL(X)$.

The non-decreasing map $p_X \colon \N \to \N$ deﬁned by $p_X (n) = |\cL_n (X)|$ is called the complexity function of X, where $\cL_n(X) = \cL(X) \cap \cA^n$. If $\displaystyle \liminf_{n \to \infty} \frac{p_X(n)}{n} = \infty$ we say that $X$ has superlinear complexity. On the contrary, $X$ has non-superlinear complexity.

Let $\cA$ and $\cB$ be finite alphabets and $\sigma\colon \cA^* \to \cB^*$ be a morphism for the concatenation. We say that $\sigma$ is \emph{erasing} whenever there exists $a \in \cA$ such that $\sigma(a)$ is the empty word. Otherwise, we say that it is \emph{non-erasing}. The morphism $\sigma$ is \emph{proper} if each word $\sigma(a)$ starts and ends with the same letter independently of $a$. When it is non-erasing it extends naturally to maps from $\cA^{\N}$ to $\cB^{\N}$ and from $\cA^{\Z}$ to $\cB^{\Z}$ in the obvious way by concatenation. To the morphism $\sigma$ we associate an \emph{incidence matrix} $M_{\sigma}$ indexed by $\cB \times \cA$ such that its entry at position $(b,a)$ is the number of occurrences of $b$ in $\sigma(a)$ for every $a \in \cA$ and $b \in \cB$ (\ie, $(M_{\sigma})_{b,a} = | \sigma(a)|_b$). If $\tau\colon \cB^* \to \cC^*$ is another morphism, then $\tau \circ \sigma\colon \cA^* \to \cC^* $ is a morphism and $M_{\tau \circ \sigma} = M_{\tau} M_{\sigma}$.  

A \emph{directive sequence} $\boldsymbol \sigma = (\sigma_n\colon \cA^*_{n+1} \to \cA^*_n )_{n \in \N}$ is a sequence of morphisms, where we only consider non-erasing ones. When all the morphisms $\sigma_n$, $n \geq 1$, are proper, we say that $\boldsymbol \sigma$ is \emph{proper}, and, when all incidence
matrices $M_{\sigma_n}$, $n \geq 1$, are positive, we say that $\boldsymbol \sigma$ is \emph{positive}. Here we stress the
fact that in the deﬁnition there is no assumption (more that non-erasingness) on the ﬁrst morphism $\sigma_0$, since in many cases of interest this morphism is neither proper nor positive. 

For $0 \leq n < N$, we denote $\sigma_{[n,N)} = \sigma_{[n,N-1]} = \sigma_n \circ \sigma_{n+1} \circ \cdots \circ \sigma_{N-1}$. 
We say that $\boldsymbol {\sigma'} = (\sigma'_k \colon \cB_{k+1}^* \to \cB_k^*)_{k \in \N} $ is a \emph{contraction} of $\boldsymbol \sigma = (\sigma_n \colon \cA_{n+1}^* \to \cA_n^*)_{n \in \N} $ if there is a sequence $(n_k)_{k \in \N}$ such that $n_0 = 0$, $\cA_{n_k} = \cB_k $ and $\sigma_k' = \sigma_{[n_k, n_{k+1})}  $ for all $k \in \N$. A directive sequence $\boldsymbol \sigma$ is \emph{primitive} if it has a positive contraction $\boldsymbol {\sigma'}$.  

For $n \in \N$, the \emph{language} $\cL^{(n)}(\boldsymbol \sigma) $ \emph{of level} $n$ \emph{associated with} $\boldsymbol \sigma $ is defined by
\begin{equation*}
    \cL^{(n)}(\boldsymbol \sigma) = \{ w \in \cA_n^* : w \sqsubseteq \sigma_{[n,N)}(a) \text{ for some } a \in \cA_N \text{ and } N>n \}
\end{equation*}
and $X_{\boldsymbol \sigma}^{(n)}$ is the set of points $x \in \cA_n^{\Z}$ such that $\cL(x) \subseteq \cL^{(n)}(\boldsymbol \sigma)$. This set is the \emph{subshift generated by } $\cL^{(n)}(\boldsymbol \sigma)$. It may happen that $\cL(X^{(n)}_{\boldsymbol \sigma}) $ is strictly contained in $\cL^{(n)}(\boldsymbol \sigma)$, but in the primitive case both sets coincide and $X^{(n)}_{\boldsymbol \sigma}$ is a minimal subshift. 
Finally, $X_{\boldsymbol \sigma} = X_{\boldsymbol \sigma}^{(0)}$ is the $\cS$-\emph{adic subshift} generated by the directive sequence $\boldsymbol \sigma$. Note that if $\boldsymbol {\sigma'}$ is a contraction of $\boldsymbol \sigma$, then $X_{\boldsymbol \sigma} = X_{\boldsymbol {\sigma'}} $.

We define the \emph{alphabet rank} of $\boldsymbol \sigma$ as 
\begin{equation*}
    AR(\boldsymbol \sigma) = \liminf_{n \to \infty} |\cA_n|.
\end{equation*}

\noindent When $AR(\boldsymbol \sigma )$ is ﬁnite, via contraction and relabeling, $(X_{\boldsymbol \sigma}, S)$ can be deﬁned by a directive sequence $\boldsymbol \sigma'$, where for every $n \geq 1$ the morphism $\sigma'_n$ is an endomorphism on a free monoid with $AR(\boldsymbol \sigma)$ generators. 

Let $\sigma \colon \cA^* \to \cB^*$ be a morphism and $x \in \cB^{\Z}$. If $x = S^k \sigma(y)$ for some $y \in \cA^{\Z} $ and $k \in \Z$, then $(k,y)$ is a $\sigma$\emph{-representation} of $x$. If $y$ belongs to some subshift $Y \subseteq \cA^{\Z}$, we say $(k,y)$ is a $\sigma$\emph{-representation} of $x$ in $Y$. Also, if $0 \leq k < |\sigma (y_0)|$ then $(k,y)$ is a \emph{centered} $\sigma$\emph{-representation} of $x$
in $Y$. Following \cite{Berthe_Steiner_Thuswaldner_Recognizability_morphism:2019}, we say that $\sigma$ is \emph{recognizable in} $Y$ if each $x \in \cB^{\Z}$ has at most one centered $\sigma$-representation in $Y$. 
If any aperiodic point $x \in \cB^{\Z}$ has at most one centered $\sigma$-representation in $Y$, we say that $\sigma$ is recognizable in $Y$ for aperiodic points. A directive sequence $\boldsymbol \sigma$ is \emph{recognizable at level} $n$ if $\sigma_n$ is recognizable in $X^{(n+1)}_{\boldsymbol \sigma}$. The sequence $\boldsymbol \sigma$ is recognizable if it is recognizable at level $n$ for each $n \in \N$.

An endomorphism $\sigma\colon \cA^* \to \cA^*$ is called a \emph{substitution}. The $\cS$-adic subshift $X_{\sigma}$ generated by the constant directive sequence $\sigma$ is determined by $ \cL( X_{\sigma})=\cL(\sigma)$, where $\cL(\sigma) = \{ w \in \cA^* : w \sqsubseteq \sigma^n(a), \text{ for some 
 } a\in \cA \text{ and } n \geq 1\}$. The subshift $(X_{\sigma},S)$ is called the \emph{substition subshift} associated with $\sigma$. A substitution is primitive if $M_{\sigma}$ is primitive.

A substitution $\sigma\colon \cA^* \to \cA^*$ has constant length if there exists a number $\ell \geq 1$ such that $|\sigma(a)| = \ell$ for all $a \in \cA$. More generally, a directive sequence $\boldsymbol \sigma = (\sigma_n\colon \cA^*_{n+1} \to \cA^*_n)_{n \in \N}$ is said to be of \emph{constant length} if there exists a sequence of positive integers $(\ell_n)_{n \in \N}$ such that $|\sigma_n(a)| = \ell_n$ for every $n \in \N$ and $a \in \cA_{n+1}$. In that case we also say that the $\cS$-adic subshift given by $\boldsymbol \sigma$ has constant length. 

Another important family of subshifts is the linearly recurrent subshifts. Using the result of Durand \cite[Proposition 1.1]{Durand2003} we say that an $\cS$-adic subshift $(X_{\boldsymbol \sigma}, S)$ is linearly recurrent if there exists a finite set of morphisms $\cS$ such that the directive sequence $\boldsymbol \sigma = (\sigma_n \colon \cA_{n+1}^* \to \cA_n^*)_{n \in \N}$ is proper, positive and $\sigma_n \in \cS$ for all $n \geq 1$. Minimal substitution subshifts and linearly recurrent subshifts are uniquely ergodic. 

For a minimal Cantor system $(X, T)$ and $U \subseteq X$ a clopen subset, we say that $(U,T_U)$ is the \emph{induced system} on $U$, where $T_Ux = T^{n(x)}x$  and $n(x) = \inf\{ k > 0 : T^kx \in U \}$, which is finite by minimality. The system $(U, T_U)$ is also a minimal Cantor system (see \cite[Chapter 1.1.3]{Durand_Perrin_Dimension_groups_dynamical_systems:2022}). If $\mu \in \cM(X,T)$ is an invariant (ergodic) measure of the original system and $\mu(U) \neq 0$, then the induced measure $\mu_U = \frac{\mu}{\mu(U)}$ is an invariant (ergodic) measure of $(U,T_U)$ (see \cite[Chapter 2]{Petersen1983}). 

\subsection{Kakutani-Rokhlin partitions}\label{subsection:towers}
Let $(X, \cX, \mu, T)$ be a measure preserving system. We say that $\sP$ is a Kakutani-Rokhlin partition of the system if it is a partition of $X$ of the form
\begin{equation*}
\sP = \{ T^j B_a : 1 \leq a \leq d, 0 \leq j < h_a \},
\end{equation*}
where $d$ is a positive integer, $B_1, \ldots, B_d$ are measurable subsets of $X$ and $h_1,\ldots,h_d$ are positive integers. 
For $a \in \{1, \ldots,d\}$, $\sT_a = \bigcup_{j=0}^{h_a - 1} T^j B_a$ is the $a$-th \emph{tower} of $\sP$, $B_a$ is the \emph{base} of this tower and $B= \bigcup_{a=1}^{d} B_a$ is the \emph{base} of $\sP$. It follows from the definition of $\sP$ that $\mu( \bigcup_{a=1} ^d \sT_a) =1$.

Now consider a sequence of Kakutani-Rokhlin partitions
\begin{equation*}
(\sP^{(n)} = \{ T^j B_a^{(n)} : 1 \leq a \leq d^{(n)}, 0 \leq j < h_a^{(n)} \} )_{n \in \N}
\end{equation*}
with $\sP^{(0)} = \{ X \}$. 
For every $n \in \N$ and $a \in \{1,\ldots,d^{(n)}\}$ we denote by $\sT^{(n)}_a$ and $B^{(n)}$ the $a$-th tower and the base of $\sP^{(n)}$ respectively.   

We say that $(\sP^{(n)})_{n \in \N}$ is \emph{nested} if for every $n \in \N$ it satisfies the following two properties:
\begin{itemize}
\item[(KR1)] $B^{(n+1)} \subseteq B^{(n)} $,
\item[(KR2)] $\sP^{(n+1)} \succeq \sP^{(n)} $; i.e., for all $A \in \sP^{(n+1)}$ there exists $A' \in \sP^{(n)}$ such that $A \subseteq A'$.
\end{itemize}

We consider mostly nested sequences of Kakutani-Rokhlin partitions which also have the following extra properties:
\begin{itemize}
\item[(KR3)] $\displaystyle \lim_{n \to \infty }\mu(B^{(n)}) = 0$,
\item[(KR4)] $\bigcup_{n \in \N} \sP^{(n)}$ generates the $\sigma$-algebra $\cX$.
\end{itemize}

\begin{remark} \label{remark KR partition}
The Jewett-Krieger theorem (see \cite{Jewett_prevalence_uniquely_ergodic:1969} and \cite{Krieger_unique_ergodicity:1972}) asserts that any measure preserving system $(X, \cX, \mu, T)$, with a non-atomic ergodic measure, is measurably isomorphic to a minimal and uniquely ergodic Cantor system. Therefore, Kakutani-Rokhlin partitions satisfying the properties (KR1)--(KR4) always exist (see, for instance, \cite{Herman1992} or \cite[Chapter 4]{Durand_Perrin_Dimension_groups_dynamical_systems:2022}). 
\end{remark}

\section{Partial rigidity in ergodic measure preserving systems} \label{sec:partial_rigidity}

This section is devoted to giving a general necessary and sufficient condition for partial rigidity (\cref{theorem KR} in the introduction). To prove this result, we will need the following lemma and some additional definitions.   

\begin{lemma}{{\cite[Lemma 4.6]{Danilenko_finite_rank_rationalerg_partial_rigidity:2016}}} \label{lema danilenko}
Let $(X, \cX, \mu, T)$ be an invertible ergodic system and $(A_n)_{n \in \N}$ a sequence in $\cX$ such that $\displaystyle\lim_{n \rightarrow \infty} \mu(A_n) = \delta >0 $ and $\displaystyle\lim_{n \rightarrow \infty}\mu (A_n \Delta TA_n) = 0$. 
Then, for every measurable set $B \in \cX$ we have 
\begin{equation*}
\mu(B \cap A_n) \xrightarrow[n \to \infty]{} \delta \mu(B).
\end{equation*} 
\end{lemma}
  
\begin{definition} \label{def subtowers}
Let $ \sP = \{ T^j B_a : 1 \leq a \leq d, 0 \leq j < h_a \}$ be a Kakutani-Rokhlin partition and $w = w_1 \ldots w_{\ell}$ be a word in the alphabet $\{1,\ldots,d\}$. We define $B_{w}$ as the subset of $X$ such that $x \in B_{w_1}$,  $T^{h_{w_1}}x \in B_{w_2}$, $T^{h_{w_2} + h_{w_1}}x \in B_{w_3}$, \ldots, $T^{h_{w_{\ell-1}} + h_{w_{\ell-2}} + \ldots + h_{w_1}}x \in B_{w_{\ell} }$. We also define the subtower generated by $w$ as $$\sT_{w} =  \bigcup_{i=0}^{h_{w_1}-1} T^i B_{w}.$$ 
In words, this is the set of points that start in tower $\sT_{w_1}$ and end in tower $\sT_{w_\ell}$, 
crossing consecutively the towers
$\sT_{w_2},\ldots,\sT_{w_{\ell-1}}$. Depending on $w$ it can be the empty set.

When we consider a sequence of Kakutani-Rokhlin partitions $(\sP^{(n)})_{n \in\N}$ we define the same objects and denote them in a natural way by: $B_w^{(n)}$ and $\sT^{(n)}_w$, where $w$ is a word in the alphabet $\{1,\ldots,d^{(n)}\}$. 
\end{definition}

\begin{definition}
Let $\sP = \{ T^j B_a : 1 \leq a \leq d, 0 \leq j < h_a \}$ be a Kakutani-Rokhlin partition and consider two complete words $w, u$ in the alphabet $\{1,\ldots,d\}$. We define the equivalence relation $w \sim_{\sP} u$ by
\begin{equation*}
\sum_{i=1}^{|w|-1} h_{w_i} =  \sum_{k=1}^{|u|-1} h_{u_k}.
\end{equation*}
We let $[w]_{\sP}$ denote the equivalence class of $w$ for this relation. Note that the last letters of each word $u$ and $w$ are not considered in the summations above, since that formula represents the number of times the transformation $T$ must be applied to follow the trajectory indicated by the words and return to the tower of origin. The idea of this equivalence class is that these numbers match.

When considering a sequence $(\sP^{(n)})_{n \in \N}$ of Kakutani-Rokhlin partitions, we write $w \sim_n u $ and $[w]_n $ instead of $w \sim_{\sP^{(n)}} u$ and $[w]_{\sP^{(n)}}$ respectively. Finally, we define
\begin{equation*}
\sT_{[w]_{\sP}} = \bigcup_{u \in [w]_{\sP}} \sT_{u} 
\end{equation*} 
and similarly $\sT_{[w]_{n}}^{(n)}$. Note that the previous union is disjoint. 
\end{definition}
As mentioned above, this equivalence class tracks when different trajectories return at the same time close to their starting point. In a sense, it is an arithmetic notion, where the different ways of expressing an integer number in the numerical base $\{h_1, \ldots, h_d\}$ are captured. 

For an alphabet $\cA$, the abelianization map $f_{\cA}: \cA^* \to \R^{\cA}$ is the function such that $f_{\cA}(w) = ( |w|_a)_{a \in \cA}$ for every $w \in \cA^*$. The following criterion,  whose proof is straightforward, provides a simple way to determine that two complete words are equivalent. 

\begin{lemma} \label{lemma abelianization}
Let $\sP = \{ T^j B_a : 1 \leq a \leq d, 0 \leq j < h_a \}$ be a Kakutani-Rokhlin partition and consider two complete words $u, w$ in the alphabet $\cA = \{1,\ldots,d\}$ such that $f_{\cA }(u_{[1,|u|-1)})= f_{\cA }(w_{[1,|w|-1)})$. Then, $u \sim_{\sP} w$.
\end{lemma}

We are now ready to prove \cref{theorem KR intro}. Recall (as mentioned in \cref{remark KR partition}) that every system $(X, \cX, \mu, T)$ with $\mu$ non-atomic and ergodic has a sequence of Kakutani-Rokhlin partitions satisfying (KR1)--(KR4).

\begin{Maintheorem}\label{theorem KR}
Let $(X, \cX, \mu, T)$ be an ergodic system with $\mu$ a non-atomic measure and let $(\sP^{(n)})_{n \in \N}$ be a sequence of Kakutani-Rokhlin partitions satisfying (KR1)-(KR4). Then, the following properties are equivalent:
\begin{enumerate}[label=\roman*)]
\item \label{teopto1} $(X, \cX, \mu, T)$ is $\gamma$-rigid.
\item \label{teopto2} There exists a sequence of complete words $(w(n))_{n \in \N}$, where $w(n) \in \{1,...,d^{(n)}\}^*$ for all $n \in \N$, such that
\begin{equation*}
\limsup_{n \to \infty} \mu( \sT^{(n)}_{[w(n)]_n}) \geq \gamma.
\end{equation*}
\end{enumerate}
\end{Maintheorem}

\begin{proof}
Assume \ref{teopto2}. Consider a sequence of complete words $(w(n))_{n \in \N}$ as in the statement \ref{teopto2} and let $\cN \subseteq \N$ be an infinite set such that
\begin{equation*} 
\lim_{n \in \cN, n \rightarrow \infty} \mu(\sT_{[w(n)]_n}^{(n)}) = \limsup_{n \rightarrow \infty} \mu(\sT_{[w(n)]_n}^{(n)}) = \delta \geq \gamma > 0.
\end{equation*}
Due to the structure of a Kakutani-Rokhlin partition and (KR3), it follows that $ \mu( \sT_{[w(n)]_n}^{(n)}  \Delta T(\sT_{[w(n)]_n}^{(n)}) ) \leq \mu(B^{(n)})$ which goes to 0 as $n$ goes to infinity.

Thus, by \cref{lema danilenko}, for any measurable set $A \in \cX$ with positive measure, we obtain that 
\begin{equation} \label{limite0}
\mu(A \cap \sT_{[w(n)]_n}^{(n)} ) \xrightarrow[n \rightarrow \infty]{n \in \cN} \delta \mu(A)  >0.
\end{equation}
    
From now on, fix $n \in \cN$, $u \in [w(n)]_n$, $a = u_1 = u_{|u|}$ and $q_n = \sum_{i=1}^{|u|-1} h_{u_i}^{(n)}$. Note that by the definition of $\sim_n$ the value $q_n$ is the same for all words in $[w(n)]_n$. Also, taking a subsequence if needed, we may assume that $(q_n)_{n \in \cN}$ is strictly increasing. Since $u$ is a complete word, by definition of $B_{u}^{(n)}$ it follows that for all $0 \leq \ell < h_a^{(n)}$,
\begin{equation} \label{superinclusion3}
T^{q_n + \ell} B_u^{(n)} \subseteq T^{\ell} B_a^{(n)}.
\end{equation}

Let $A \in \cX$ be a set of positive measure of the form
\begin{equation} \label{conjA}
A = \bigcup_{l = 1 }^{d^{(n_0)}} A_l = \bigcup_{l = 1 }^{d^{(n_0)}} \bigcup_{j \in J_l} T^j B_l^{(n_0)},
\end{equation}
where $n_0 < n $, $J_l \subseteq \{ 0,\ldots,h_l^{(n_0)} -1\}$ and the set $A_l$ is the disjoint union of some floors of the tower $\sT^{(n_0)}_l$, for all $l \in \{1, \ldots, d^{(n_0)}\}$. If we call $D_{u} = A \cap \sT_{u}^{(n)}$, by (\ref{limite0}) we can assume that $\mu(D_{u}) >0$, otherwise we take a larger $n$ or another word $u' \in [w(n)]_n$. 
Since Kakutani-Rokhlin partitions are nested, then for all $0 \leq \ell < h_a^{(n)}$ and $l \in \{1,\ldots,d^{(n_0)}\}$, $T^{\ell} B_a^{(n)}$ is either disjoint from $T^j B_l^{(n_0)}$ (for $j \in J_l$) or is included. Thus, if $x \in D_{u}$, then there exist $0 \leq \ell < h_a^{(n)}$, $l \in \{1,\ldots,d^{(n_0)}\}$ and $j \in J_l$ such that $x \in T^{\ell} B_a^{(n)}$ and $x \in T^jB_u^{(n_0)}$, \ie, they are not disjoint. Therefore, $T^{\ell} B_a^{(n)} \subseteq T^j B_l^{(n_0)}$. 

Then, using \eqref{superinclusion3}, it follows that if $x \in D_{u}$, then $T^{q_n}x \in T^{q_n + \ell} B_{u}^{(n)} \subseteq  T^{\ell} B_a^{(n)} \subseteq T^jB_l^{(n_0)} \subseteq A $. Therefore, $D_{u} \subseteq A \cap T^{-q_n} A$. Since this holds for any $u \in [w(n)]_n$ with $\mu(D_{u}) >0 $, by setting  $D_n = \bigcup_{u \in [w(n)]_n} D_{u} = A \cap \sT_{[w(n)]_n} ^{(n)},$
and thanks to (\ref{limite0}), we obtain that
\begin{equation} \label{eq_partial_A}
 \liminf_{n\to \infty} \mu (A \cap T^{-q_n} A) \geq \liminf_{n\to \infty} \mu (D_n) = \delta \mu(A) \geq \gamma \mu(A).
\end{equation}

We conclude by noting that by (KR4) the collection of sets as in \eqref{conjA} generates $\cX$ and standard approximation arguments allow us to extend \eqref{eq_partial_A} to any measurable set. This shows \ref{teopto1}.

Now assume \ref{teopto1}.  Let $(n_k)_{k\in \N}$ be a sequence associated with $\gamma$-rigidity.
Fix $\varepsilon >0$ and $m \geq 1$. By assumption, there exists $k_{\varepsilon,m} \in \N$ such that for all $a \in \{ 1,\ldots,d^{(m)}\}$
\begin{equation} \label{desigualdad01}
    \mu(B_a^{(m)} \cap T^{-n_k} B_a^{(m)} ) \geq \gamma \mu(B_a^{(m)}) - \frac{\varepsilon}{p_m} \qquad \forall k \geq k_{\varepsilon,m} ,
\end{equation}
where $p_m = \sum_{l = 1}^{d^{(m)}} h_l^{(m)}$.

Then observe that $x \in B_a^{(m)} \cap T^{-n_k} B_a^{(m)}$ if and only if there exists a complete word $w = w_1  \ldots w_{\ell} $ such that $w_1 = w_{\ell} = a$, $h^{(m)}_{w_1} + \cdots + h^{(m)}_{w_{\ell-1}} = n_k$ and $x \in B_{w}(m)$. Note that all words that satisfy this condition are equivalent, and we write $[w]_m$ to denote this class (here $w$ depends on $m$ and $n_k$ but we omit writing this dependency for the sake of brevity of notation). It follows that,
\begin{equation*}
    B_a^{(m)} \cap T^{-n_k} B_a^{(m)} \subseteq \bigcup_{\substack{u \in [w]_m \\ u_1 = a }} B_u^{(m)}.
\end{equation*}

We get,
\begin{align*}
\mu\left(\bigcup_{\substack{u \in [w]_m \\ u_1 = a }} \sT_{u}^{(m)}\right) 
    &= \sum_{\substack{u \in [w]_m \\ u_1 = a }}\mu( \sT_{u}^{(m)}) 
    = \sum_{\substack{u \in [w]_m \\ u_1 = a }} h_a^{(m)} \mu( B_u^{(m)}) \\
    &= h_a^{(m)} \mu\left(\bigcup_{\substack{u \in [w]_m \\ u_1 = a }}  B_u^{(m)}\right) 
    \geq  h_a^{(m)} \mu( B_a^{(m)} \cap T^{-n_k} B_a^{(m)} )  \\
    &\geq  h_a^{(m)} \left(\gamma \mu(B_a^{(m)}) - \frac{\varepsilon}{p_m} \right) = \gamma \mu(\sT_a^{(m)}) - h_a^{(m)} \frac{\varepsilon}{p_m}.
\end{align*}
Since $a \in \{1,\ldots,d^{(m)}\}$ is arbitrary, it follows that
\begin{align*}
    \mu\left( \sT_{[w]_m}^{(m)}\right)&= \mu\left(\bigcup_{u \in [w]_m} \sT_{u}^{(m)}\right) = \sum_{a =1}^{d^{(m)}}  \mu\left(\bigcup_{\substack{u \in [w]_m \\ u_1 = a }}  \sT_{u}^{(m)}\right) \\
    &\geq \sum_{a =1}^{d^{(m)}} \left( \gamma \mu(\sT_a^{(m)}) - h_a^{(m)} \frac{\varepsilon}{p_m} \right) = \gamma - \varepsilon. 
\end{align*}

We get
\begin{equation*}
\limsup_{m \rightarrow \infty} \mu(\sT^{(m)}_{[w(m)]_m}) \geq \gamma - \varepsilon.
\end{equation*}
Since $\varepsilon > 0$ was arbitrary, we conclude \ref{teopto2}. 
\end{proof} 
In particular, \cref{theorem KR} provides the following characterization of rigidity:

\begin{corollary} \label{cor equivrigid}
Let $(X, \cX, \mu, T)$ be an ergodic system with $\mu$ a non-atomic measure and let $(\sP^{(n)})_{n \in \N}$ be a sequence of Kakutani-Rokhlin partitions satisfying (KR1)--(KR4). Then, the following properties are equivalent:
\begin{enumerate}[label=\roman*)]
\item \label{corpto1} $(X, \cX, \mu, T)$ is rigid.
\item \label{corpto2} There exists a sequence of complete words $(w(n))_{n \in \N}$, where for all $n\in \N$ words $w(n)$ are in $\{1,\ldots,d^{(n)}\}^*$, such that
\begin{equation*}
\limsup_{n \to \infty} \mu( \sT^{(n)}_{[w(n)]_n}) = 1.
\end{equation*}
\end{enumerate}
\end{corollary}

\begin{remark} \label{remark partial seq}
In \cref{theorem KR} (resp. \cref{cor equivrigid})  any increasing subsequence of $n_k=\sum_{i=1}^{|w(k)| -1} h^{(k)}_{w_i(k)}$ is a partial rigidity sequence (resp. the rigidity sequence).
\end{remark}

We finish this section by stating a useful corollary which provides a sufficient condition for partial rigidity. Versions of this condition have been used implicitly in other papers (see, for instance, \cite{Bezuglyi_Kwiatkowski_Medynets_Solomyak_Finite_rank_Bratteli:2013,Danilenko_finite_rank_rationalerg_partial_rigidity:2016}). Given a word $w$, it might be difficult to determine which words are in its equivalence class, so the following corollary is easier to verify as it considers towers given just by one word. 

\begin{corollary} \label{cor KR}
Let $(X, \cX, \mu, T)$ be an ergodic system with $\mu$ a non-atomic measure and let $(\sP^{(n)})_{n \in \N}$ be a sequence of Kakutani-Rokhlin partitions satisfying (KR1)--(KR4).
If there exists a sequence of complete words $(w(n))_{n \in \N}$ associated to $(\sP^{(n)})_{n \in \N}$, where $w(n) \in \{1,\ldots,d^{(n)}\}^*$ for all $n \in \N$, such that
\begin{equation*}
\limsup_{n \to \infty} \mu( \sT^{(n)}_{w(n)}) \geq \gamma,
\end{equation*}
then, $(X, \cX, \mu, T)$ is $\gamma$-rigid.
\end{corollary}

\begin{proof}
It follows directly from the fact that $\mu( \sT^{(n)}_{w(n)}) \leq \mu( \sT^{(n)}_{[w(n)]_n})$ and \cref{theorem KR}.
\end{proof}

\section{Partial rigidity rate} \label{sec:partial_rigid_rate}

When studying partial rigidity, it is natural to ask about the best partial rigidity constant. More precisely, 
\begin{definition}
    Let $(X, \cX, \mu, T)$ be a measure preserving system. The partial rigidity rate of $(X, \cX, \mu, T)$ is given by
    \begin{equation*}
        \delta_{\mu} = \sup \{ \gamma \in (0,1] : \mu \text{ is } \gamma\text{-rigid} \},
    \end{equation*}
where we set $\delta_{\mu}=0$ if $(X, \cX, \mu, T)$ is not partially rigid.
\end{definition}

Notice that although we sometimes refer to the partial rigidity rate of a topological dynamical system, we are always talking about the system endowed with a particular invariant measure, since the definition of partial rigidity is a purely measure-theoretic concept. 

It is clear from the definition that $(X, \cX, \mu, T)$ is $\gamma$-rigid for every $\gamma < \delta_{\mu}$. We will see below the less obvious fact that $(X, \cX, \mu, T)$ is also $\delta_{\mu}$-rigid.  
We first state a set of general properties related to the partial rigidity rate, some of which have already been outlined in \cite[Proposition 1.13]{King_joining-rank_finite_mixing:1988}. Recall that for a sequence of systems $ \left ( (X_i, \cX_i,\mu_i, T_i) \right ) _{i \in \N}$ such that for every $i \in \N$ there is a factor map $\pi_i \colon X_{i+1} \to X_i $, its inverse limit $\varprojlim(X_i, \cX_i,\mu_i, T_i)$ is the system $\displaystyle (Z,\mathcal{Z},\rho,R)$ given by
$Z = \{ (x_i)_{i \in \N} \in \prod_{i\in \N} X_i : \pi_i(x_{i+1}) = x_i, i \in \N\}$, $\mathcal{Z}$ the smallest $\sigma$-algebra that makes all coordinate projections $p_j\colon Z\to X_i$ measurable and $\rho$ the measure defined by $\rho(p_i^{-1}(A))=\mu_i(A))$ for all $i\in \N$ and $A\in \cX_i$. The transformation $R$ is given by $R((x_i)_{i \in \N}) = (T_i(x_i))_{i \in \N} $.

\begin{proposition} \label{prop:prop_parcial_rigid}
 Let $\pi\colon X\to Y$ be a factor map between the measure preserving systems $(X, \cX, \mu, T)$ and $(Y, \cY, \nu, S)$. 

\begin{enumerate}
    \item  If $(X, \cX, \mu, T)$ is $\gamma$-rigid for the sequence $(n_k)_{k \in \N}$, then $(Y, \cY, \nu, S)$ is $\gamma$-rigid for the same sequence and $\delta_{\mu} \leq \delta_{\nu}$. In particular, if $(X, \cX, \mu, T)$ and $(Y, \cY, \nu, S)$ are isomorphic, then $\delta_{\mu} = \delta_{\nu}$.
    \item For the product system $\displaystyle(X^n, \otimes_{i=1}^n \cX, \mu^n, T \times \cdots \times T) $, we have $\delta_{\mu^n} = \delta_{\mu}^n$.
    \item If $\displaystyle (Z, \mathcal{Z}, \rho,R) = \varprojlim (X_i,\cX_i,\mu_i, T_i)$, then $\delta_{\rho} = \inf_{i \in \N} \delta_{\mu_i} = \lim_{i \to \infty} \delta_{\mu_i}$.    
\end{enumerate}

\end{proposition}

To show that the partial rigidity rate is indeed a partial rigidity constant, we need the following lemma. 

\begin{lemma}  \label{teodeltaX}
Let $(X, \cX, \mu, T)$ be a partially rigid ergodic system with $\mu$ a non-atomic measure and let $(\sP^{(n)})_{n \in \N}$ be a sequence of Kakutani-Rokhlin partitions satisfying (KR1)--(KR4). Then 
\begin{equation} \label{eqdX}
   \delta_{\mu} = \limsup_{n \rightarrow \infty} \left\{ \sup_{\substack{w \in \{1,...,d^{(n)}\}^* \\ w_1 = w_{|w|} }}  \mu ( \sT_{[w]_n}^{(n)} ) \right\}.
\end{equation}
Furthermore, $(X,\cX,\mu,T)$ is $\delta_{\mu}$-rigid. 
\end{lemma}
\begin{proof} Let $M$ be the right-hand side in (\ref{eqdX}). Every sequence of complete words $(w(n))_{n \in \N} $, with $w(n) \in \{1,...,d^{(n)}\}^*$, satisfies $ \limsup_{n \rightarrow \infty} \mu ( \sT^{(n)}_{[w (n)]_n} ) \leq M$ and therefore, by \cref{theorem KR}, any constant of partial rigidity is bounded by $M$. That is, $\delta_{\mu}\leq M$.
Conversely, let $u(n) \in \{ 1, \ldots, d^{(n)}\}^*$ be a sequence of complete words such that for every $n \in \N$,
$\mu ( \sT^{(n)}_{[u (n)]_n} ) \geq \sup_{\substack{w \in \{1,\ldots,d^{(n)}\}^* \\ w_1 = w_{|w|} }}  \mu ( \sT_{[w]_n}^{(n)} ) - \frac{1}{n}$. 
Then,
\begin{align*}
    M &\geq \limsup_{n \rightarrow \infty} \mu ( \sT^{(n)}_{[u (n)]_n} ) \\
    &\geq \limsup_{n \rightarrow \infty} \left\{ \sup_{\substack{w \in \{1,\ldots,d^{(n)}\}^* \\ w_1 = w_{|w|} }}  \mu ( \sT_{[w]_n}^{(n)} ) - \frac{1}{n} \right\} = M.
\end{align*}

We conclude, thanks to \cref{theorem KR}, that $(X,\cX,\mu,T)$ is $M$-rigid. 
Thus, $\delta_{\mu} =M$ and the system is $\delta_{\mu}$-rigid.
\end{proof} 

We now state the main result of this section, namely the description of the partial rigidity rate $\delta_{\mu}$.

\begin{theorem} \label{theorem partial rigidity rate}
Let $(X, \cX, \mu, T)$ be an ergodic and partially rigid system. Then, $(X, \cX, \mu, T)$ is $\delta_{\mu}$-rigid. Moreover, if $\mu$ is non-atomic and $(\sP^{(n)})_{n \in \N}$ is a sequence of Kakutani-Rokhlin partitions satisfying (KR1)--(KR4), then 
\begin{equation} \label{eqdX 2}
\delta_{\mu} = \inf_{n \geq 1} \left\{ \sup_{\substack{w \in \{1,\ldots,d^{(n)}\}^* \\ w_1 = w_{|w|} }}  \mu ( \sT_{[w]_n}^{(n)} ) \right\}.
\end{equation}

\end{theorem}
\begin{proof} Notice that the atomic (periodic) case is trivial (see \cref{remark periodic}), so we only consider the non-atomic case. By \cref{teodeltaX}, the system is $\delta_{\mu}$-rigid and so we are left to show that
\begin{equation*}
\limsup_{n \rightarrow \infty} \left\{ \sup_{\substack{w \in \{1,\ldots,d^{(n)}\}^* \\ w_1 = w_{|w|} }}  \mu ( \sT_{[w]_n}^{(n)} ) \right\} 
 =  \inf_{n \geq 1} \left\{ \sup_{\substack{w \in \{1,\ldots,d^{(n)}\}^* \\ w_1 = w_{|w|} }}   \mu ( \sT_{[w]_n}^{(n)} ) \right\},
\end{equation*}
and then use \cref{teodeltaX} once again to conclude (\ref{eqdX 2}). 
     
To this end, since $(\sP^{(n)})_{n \in \N}$ is nested, if $a \in \{1, \ldots, d^{(n+1)} \}$, then each floor of the tower $\sT^{(n+1)}_a$ is a subset of a tower $\sT^{(n)}_l$ with $l \in \{1,\ldots,d^{(n)}\}$. 
We define the morphism $\tau_n \colon \{1,\ldots, d^{(n+1)} \} ^* \to \{1, \ldots, d^{(n)} \} ^*$ as $\tau_n(a) = w^a = w_1^a \ldots w_{\ell_a}^a$ such that $B_a^{(n+1)} \subseteq B_{w_1^a}^{(n)}$, $T^{h_{w_1^a}^{(n)}}B_a^{(n+1)} \subseteq B_{w_2^a}^{(n)},\ldots, T^{h_{w_1^a}^{(n)} + \cdots + h_{w_{\ell_a -1}^a}^{(n)}}B_a^{(n+1)} \subseteq B_{w_{\ell_a}^a}^{(n)}$ and $h_{w_1^a}^{(n)} + \cdots + h_{w_{\ell_a}^a}^{(n)} = h^{(n+1)}_{a}$. 
     
Thus, if $w =  a_1 \ldots a_{m} \in \{1,\ldots, d^{(n+1)} \}^*$ is a complete word, we have that if $x \in \sT_{a_1}^{(n+1)}$ then there exists $0 \leq k <h_{w^{a_1}_{j}}^{(n)}$ such that $ x \in T^{k+h_{w^{a_1}_{j-1}}^{(n)} } B_{a_1}^{(n+1)} \subseteq T^k B_{w^{a_1}_j}^{(n)}$ (for $0< j \leq \ell_{a_1}$), \ie,  $x\in \sT^{(n)}_{w^{a_1}_j}$. Furthermore, $x \in \sT_{s_{j}}^{(n)}$, where $s_{j} = w^{a_1}_{[j,\ell_{a_1}]}$. 
More generally, when $x \in \sT^{(n+1)}_w$ we also know that $x$ visits the towers $\sT^{(n+1)}_{a_2}, \ldots, \sT^{(n+1)}_{a_m}$ in that order, so $x$ visits $\sT^{(n)}_{w^{a_2}_1}, \ldots, \sT^{(n)}_{w^{a_2}_{\ell_{a_2}}}, \ldots,\sT^{(n)}_{w^{a_m}_1}, \ldots, \sT^{(n)}_{w^{a_m}_{\ell_{a_m}}}.$ Therefore, $x \in \sT_{s_j w^{a_2} \ldots w^{a_{m-1}} w^{a_m}}^{(n)}$. In particular, $x \in \sT_{s_j w^{a_2} \ldots w^{a_{m-1}} p_j}^{(n)}$, where $p_j= w^{a_m}_{[1, j]}= w^{a_1}_{[1, j]}$. Therefore,
\begin{equation} \label{eq inclusion 1}
\sT_w^{(n+1)} \subseteq \bigcup_{j \in \{1,\ldots, \ell_{a_1}  \}  } \sT_{s_j w^{a_2} \ldots w^{a_{m-1}} p_j}^{(n)}.
\end{equation}
If $v(j) = s_j w^{a_2} \ldots w^{a_{m-1}} p_j$, it is easy to check that it is a complete word with first and last letters equal to $w^{a_1}_j$. Moreover, for every $j,j' \in \{1,\ldots,\ell_{a_1}\}$, $v(j) \sim_n v(j')$, which is a consequence of the fact that 

\begin{align*}
    \sum_{k=1}^{|v(j)|-1} h^{(n)}_{v(j)_k} &= \sum_{i=1}^{|s_j|} h^{(n)}_{w^{a_1}_{j-1+i}} + \sum_{i=1}^{|w^{a_2}|} h^{(n)}_{w^{a_2}_i}  + \ldots + \sum_{i=1}^{|w^{a_{m-1}}|} h^{(n)}_{w^{a_{m-1}}_i} + \sum_{i=1}^{|p_j|-1} h^{(n)}_{w^{a_1}_{i}} \\
    &= \sum_{i=1}^{|w^{a_1}|} h^{(n)}_{w^{a_1}_i} + \sum_{i=1}^{|w^{a_2}|} h^{(n)}_{w^{a_2}_i}  + \ldots + \sum_{i=1}^{|w^{a_{m-1}}|} h^{(n)}_{w^{a_{m-1}}_i}\\
    &=   h^{(n+1)}_{a_1} + \ldots + h^{(n+1)}_{a_{m-1}}.
\end{align*}
This clearly implies that $\displaystyle \sum_{k=1}^{|v(j)|-1} h^{(n)}_{v(j)_k} = \sum_{k=1}^{|v(j')|-1} h^{(n)}_{v(j')_k}$ for every $j,j' \in \{1,\ldots,\ell_{a_1}\}$. In order to avoid any confusion with the above calculation, it is worth recalling that the summation that defines the equivalence relation $\sim_n$ does not take into account the height of the tower given by the last letter. 

Therefore, by (\ref{eq inclusion 1}):
\begin{equation*}
\sT_w^{(n+1)} \subseteq \sT_{[v(1)]_n}^{(n)}.
\end{equation*}

Repeating the argument with another complete word $w' \in [w]_{n+1}$ we conclude that for $v'(1)$ constructed as above, $\sT_{w'}^{(n+1)} \subseteq \sT_{[v'(1)]_n}^{(n)}$. Moreover, as $w \sim_{n+1} w'$, using the above equalities, we have that $v'(1) \sim_n v(1)$ because
\begin{align*}
    \sum_{k=1}^{|v(j)|-1} h^{(n)}_{v(j)_k} = \sum_{k=1}^{|w|-1}  h^{(n+1)}_{a_k} = \sum_{k=1}^{|w'|-1}  h^{(n+1)}_{a'_k} =  \sum_{k=1}^{|v'(j)|-1} h^{(n)}_{v'(j)_k} .
\end{align*}
Thus, we conclude that
    \begin{equation*}
        \sT_{[w]_{n+1}}^{(n+1)} \subseteq  \sT_{[v(1)]_n}^{(n)}.
    \end{equation*}

This implies that
    \begin{equation*}
        \sup_{\substack{w \in \{1,\ldots,d^{(n+1)}\}^* \\ w_1 = w_{|w|} }}  \mu ( \sT_{[w]_{n+1}}^{(n+1)} ) \leq  \sup_{\substack{v \in \{1,\ldots,d^{(n)}\}^* \\ v_1 = v_{|v|} }}  \mu ( \sT_{[v]_n}^{(n)} ), 
    \end{equation*}
which means that the sequence of supremums is decreasing, so the limit exists and is equal to the infimum.
\end{proof}

\begin{remark}
If the measure preserving system is not partially rigid, then\break  $\displaystyle \inf_{n \geq 1}  \sup_{\substack{w \in \{1,\ldots,d^{(n)}\}^* \\ w_1 = w_{|w|} }}  \mu ( \sT_{[w]_n}^{(n)} ) =0$, otherwise \cref{theorem KR} would imply that the system is partially rigid. So, the equality \eqref{eqdX 2} actually holds for any measure preserving system. 
\end{remark}

\begin{remark}
 \cref{theorem partial rigidity rate intro} stated in the introduction follows from \cref{theorem partial rigidity rate}. Indeed, the sequence of complete words $(u(n))_{n \in \N}$ that appears in the proof of \cref{teodeltaX} fulfills that
    \begin{equation*} 
\sup_{\substack{w \in \{1,\ldots,d^{(n)}\}^* \\ w_1 = w_{|w|} }}  \mu ( \sT_{[w]_n}^{(n)} ) \leq \mu ( \sT^{(n)}_{[u (n)]_n} ) + \frac{1}{n}, 
\end{equation*}
and then, from \cref{theorem partial rigidity rate}:

\begin{align*}
    \delta_{\mu} = \inf_{n \geq 1} \left\{ \sup_{\substack{w \in \{1,\ldots,d^{(n)}\}^* \\ w_1 = w_{|w|} }}  \mu ( \sT_{[w]_n}^{(n)} ) \right\} & \leq \liminf_{n \to \infty}   \mu ( \sT^{(n)}_{[u (n)]_n} ) + \frac{1}{n} \\ 
    &\leq \limsup_{n \to \infty}  \mu ( \sT^{(n)}_{[u (n)]_n} )  \leq \delta_{\mu}.
\end{align*}
Therefore, $\displaystyle \lim_{n \to \infty} \mu(\sT^{(n)}_{[u(n)]_n}) = \lim_{n \to \infty} \sum_{w \sim_n u(n)} \mu(\sT^{(n)}_{w}) = \delta_{\mu}$.
\end{remark}

\begin{remark}
   In the $\cS$-adic case, which will be studied in the following sections, 
   the morphism constructed in the previous proof will be precisely the morphism of the directive sequence $\boldsymbol{\sigma}$ defining the subshift $X_{\boldsymbol \sigma}$. 
\end{remark}

\begin{remark} \label{remark non mixing non p rigid}
    Now we are able to construct an ergodic, non-mixing and non-partially rigid system. First, notice that \cref{theorem partial rigidity rate} implies that $(X, \cX, \mu, T)$ is not rigid if and only if $\delta_{\mu} < 1$. 
    Therefore, if $(X, \cX, \mu, T)$ is weak mixing but it is not mixing (for example, the Chacon subshift), then the countable product $(X^{\N}, \cX^{\N}, \mu^{\N}, T \times T \times \cdots )$ is an ergodic non-mixing system such that $\delta_{\mu^{\N}} = \displaystyle \lim_{n \to \infty} \delta_{\mu}^n = 0$ and so it is not partially rigid. The existence of such systems is already mentioned in a remark after Proposition 1.13 in King's paper \cite{King_joining-rank_finite_mixing:1988}.
\end{remark}

\section{Partial rigidity in $\cS$-adic subshifts } \label{section partial rigidity Sadic}

In this section, we focus on the study of partial rigidity for $\cS$-adic subshifts. From results of Creutz \cite{Creutz_mixing_minimal_comp:2023}, we know that non-superlinear complexity subshifts, which by \cite{Donoso_Durand_Maass_Petite_interplay_finite_rank_Sadic:2021} are $\cS$-adic of finite alphabet rank, are partially rigid. However, there are no general conditions that allow us to deduce this property for an arbitrary $\cS$-adic subshift, let alone provide sequences of integers giving partial rigidity. In this section, we exploit the combinatorial structure inherent to each $\cS$-adic subshift to deduce from the results of previous sections several sufficient conditions that imply partial rigidity.

\subsection{Kakutani-Rokhlin partitions in $\cS$-adic subshifts}

Let $\boldsymbol \sigma = (\sigma_n \colon \cA_{n+1}^* \to \cA_n^*)_{n \in \N}$ be a primitive recognizable directive sequence, and let $(X_{\boldsymbol \sigma}, S)$ be the $\cS$-adic subshift generated by $\boldsymbol \sigma$. We define its natural sequence of Kakutani-Rohklin partitions $(\sP^{(n)})_{n \in \N}$ as follows:
\begin{equation*}
    \sP^{(n)} =  \{ S^k \sigma_{[0,n)}([a]) : a \in \cA_n, 0 \leq k < \lvert \sigma_{[0,n)} (a) \rvert \}.
\end{equation*}

It is straightforward from \cite[Lemma 6.3]{Berthe_Steiner_Thuswaldner_Recognizability_morphism:2019} that if $\mu$ is an $S$-invariant probability measure and we consider the measure preserving system $(X_{\boldsymbol \sigma}, \cB, \mu,S)$, then $(\sP^{(n)})_{n \in \N}$ satisfies (KR1)--(KR4) (here $\cB$ is the Borel sigma-algebra). 
Moreover, sets in $\sP^{(n)}$ are not only measurable, but also clopen and, when the directive sequence $\boldsymbol \sigma$ is proper, then $\bigcap_{n \geq 1} B^{(n)}$ is equal to a single point and $\bigcup_{n \in \N} \sP^{(n)}$ generates the topology (see \cite[Chapter 5.3]{Durand_Perrin_Dimension_groups_dynamical_systems:2022} for more details).  Another important aspect is that the induced system on the base $B^{(n)}$ coincides with $X^{(n)}_{\boldsymbol \sigma}$. So each invariant measure on $X^{(n)}_{\boldsymbol \sigma}$ can be treated as the induced measure $\mu(\cdot)/\mu(B^{(n)})$ of an invariant measure $\mu$ on $X_{\boldsymbol \sigma}$.

If $\boldsymbol \sigma$ is primitive and recognizable, for any $w \in \cA_{n}^*$, $\sT^{(n)}_w$ has a positive measure if and only if $w \in \cL^{(n)}(\boldsymbol \sigma)$. 
 This follows from the fact that $B_w^{(n)}$ coincides with the cylinder set $[w]_{X_{\boldsymbol \sigma}^{(n)}}$ on the induced system on $B^{(n)}$ and $[w]_{X_{\boldsymbol \sigma}^{(n)}}$ has positive measure whenever $w \in \cL^{(n)}(\boldsymbol \sigma)$.

The following two lemmas will be very useful in the next sections.

\begin{lemma} \label{lemma torres w}
Let $\boldsymbol \sigma = ( \sigma_n \colon \cA_{n+1}^* \to \cA_n^*)_{n \in \N}$ be a primitive and recognizable directive sequence. For $a \in \cA_{n+1}$, if $\sigma_n(a) = p w^r s$ with $r \geq 1$, $p,w,s \in \cA_n^*$ and $w_1 = b$, then 
\begin{equation*}
\mu(\sT^{(n)}_w \cap \sT^{(n+1)}_a) \geq \frac{r}{|\sigma_n(a)|_b} \mu(\sT^{(n)}_b \cap \sT^{(n+1)}_a)
\end{equation*}
and 
\begin{equation*}
\mu(\sT^{(n)}_{wb} \cap \sT^{(n+1)}_a) \geq \frac{r-1}{|\sigma_n(a)|_b} \mu(\sT^{(n)}_b \cap \sT^{(n+1)}_a).
\end{equation*}
\end{lemma}

\begin{proof}
It is clear that $\mu(\sT^{(n)}_b \cap \sT^{(n+1)}_a) = |\sigma_n(a)|_b h_b^{(n)} \mu(B_a^{(n+1)})$ 
and that for every $j \in \{0,\ldots,r-1\}$, $S^{\sum_{i =1}^{|p|} h_{p_i}^{(n)} + j \sum_{i=1}^{|w|} h_{w_i}^{(n)} } B_a^{(n+1)} \subseteq B_w^{(n)}$. 
Therefore, $\mu(\sT^{(n)}_w \cap \sT^{(n+1)}_a) \geq r h_b^{(n)} \mu(B_a^{(n+1)})$ and
\begin{equation*}
\mu(\sT^{(n)}_w \cap \sT^{(n+1)}_a) \geq r \frac{|\sigma_n(a)|_b}{|\sigma_n(a)|_b} h_b^{(n)} \mu(B_a^{(n+1)}) = \frac{r}{|\sigma_n(a)|_b} \mu(\sT^{(n)}_b \cap \sT^{(n+1)}_a).
\end{equation*}
For the second inequality, the argument is analogous.
\end{proof}

More generally,

\begin{lemma} \label{lemma torres w2}
Let $\boldsymbol \sigma = ( \sigma_n \colon \cA_{n+1}^* \to \cA_n^*)_{n \in \N}$ be a primitive and recognizable directive sequence. 
For $a \in \cA_{n+1}$ and $w \in \cA_{n}^*$ with $w_1 = b$ we have that 
\begin{equation} \label{inequality towers w}
\mu(\sT^{(n)}_w \cap \sT^{(n+1)}_a) \geq \frac{|\sigma_n(a)|_w}{|\sigma_n(a)|_b} \mu(\sT^{(n)}_b \cap \sT^{(n+1)}_a).
\end{equation}
\end{lemma}
\begin{proof}
Fix $n \in \N$, $a \in \cA_{n+1}$, $b \in \cA_n$ and $w \in \cA_n^*$ as above. Let $I$ be the set of occurrences of the word $w$ in $u = \sigma_n(a)$, that is, $I = \{ i \in \{ 1, \ldots, |u|\} : u_{[i,i+|w|)} = w \}$. 
Thus, for every $i \in I$, $S^{\sum_{j =1}^{i-1} h_{u_j}^{(n)} } B_a^{(n+1)} \subseteq B_w^{(n)}$. 
Therefore, $\mu(\sT^{(n)}_w \cap \sT^{(n+1)}_a) \geq |I| h_b^{(n)} \mu(B_a^{(n+1)})$ and $|I| = |\sigma_n(a)|_w$. We get that,
\begin{equation*}
\mu(\sT^{(n)}_w \cap \sT^{(n+1)}_a) \geq |\sigma_n(a)|_w \frac{|\sigma_n(a)|_b}{|\sigma_n(a)|_b} h_b^{(n)} \mu(B_a^{(n+1)}) = \frac{|\sigma_n(a)|_w}{|\sigma_n(a)|_b} \mu(\sT^{(n)}_b \cap \sT^{(n+1)}_a).
\end{equation*}
\end{proof}

\begin{remark}
Inequality (\ref{inequality towers w}) could be strict. For example, consider the substitution subshift associated with the substitution $\sigma\colon \{0,1\}^* \to \{0,1\}^*$ such that $\sigma(0)= 10111001$ and $\sigma(1)=10001$.  

If $w = 11$ then $|\sigma(0)|_{11} = 2 $ and so $ \frac{|\sigma(0)|_{11} }{|\sigma(0)|_{1} } = \frac{2}{5}$. 
 Also, since the substitution is left proper with $\sigma(0)_1 = \sigma(1)_1 = 1$, then on $\sigma(0)\sigma(c)$
the last $1$ from $\sigma(0)$ is always followed by another $1$, for any $c\in\{0,1\}$,  and  $S^{ 3 h^{(n)}_0 + 4 h^{(n)}_1} B_0(n+1)$ is included in $B_{11}(n)$.
Therefore, $\mu(\sT^{(n)}_{11} \cap \sT_0^{(n+1)}) = \frac{3}{5} \mu(\sT^{(n)}_{1} \cap \sT_0^{(n+1)}) > \frac{2}{5} \mu(\sT^{(n)}_{1} \cap \sT_0^{(n+1)})$. 

Using the same reasoning, choosing the substitution given by $\tau\colon \{0,1\}^* \to \{0,1\}^*$ such that
$\tau(0)=0111001$ and $\tau(1)=0001$,
if $w=11$ and $a=0$, the inequality (\ref{inequality towers w}) is, in fact, an equality. 

In general, the smaller the quotient $|w|/|\sigma_n(a)|$ is, the closer the inequality (\ref{inequality towers w}) is to equality. 
\end{remark}

Finally, we define a \emph{clean} directive sequence. This notion has different versions in the literature (see, for example, \cite{Bezuglyi_Kwiatkowski_Medynets_Solomyak_Finite_rank_Bratteli:2013,Bressaud_Durand_Maass_Eigenvalues_finite_rank:2010} or \cite{Arbulu_Durand_properties_Ferenczi_subshifts:2022} in the $\cS$-adic context). 
Let $\boldsymbol \sigma $ be a primitive and recognizable directive sequence, and let $\mu$ be an ergodic measure of $(X_{\boldsymbol \sigma},S)$. We say that $\boldsymbol \sigma $ is \emph{clean} with respect to $\mu$ if:
\begin{enumerate}
\item There exists $n_0 \in \N$ such that $\cA_n = \cA_{n_0}$ for all $n \geq n_0$. Set $\cA = \cA_{n_0}$;
\item There exists a constant $\eta >0 $ and $\cA_{\mu} \subseteq \cA$ such that:
\begin{align*}
\mu(\sT^{(n)}_a) \geq \eta & \quad \forall a \in \cA_{\mu}, n \geq n_0 \quad \text{ and } \\
\lim_{n \to \infty } \mu(\sT^{(n)}_a) = 0 & \quad \forall a \in \cA \backslash \cA_{\mu}. 
\end{align*}
\end{enumerate}
Combining the results from \cite{Bezuglyi_Kwiatkowski_Medynets_Solomyak_Finite_rank_Bratteli:2013} and \cite{Berthe_Steiner_Thuswaldner_Recognizability_morphism:2019}, it is known that any recognizable and primitive directive sequence $\boldsymbol \sigma $ of finite alphabet rank can be contracted in order to become \emph{clean} with respect to $\mu$. More generally, we say that $\boldsymbol \sigma $ is \emph{clean}, if it is clean with respect to every ergodic measure $\mu$. When $\boldsymbol \sigma$ is clean, if $\mu$ and $\nu$ are distinct ergodic measures of $(X_{\boldsymbol \sigma},S)$, then $\cA_{\mu} \cap \cA_{\nu} = \emptyset$. 

\subsection{Sufficient conditions for partial rigidity} \label{subsec:conditions_partial_rigidity}

In the following, we provide a set of sufficient conditions that guarantee partial rigidity of an $\mathcal{S}$-adic subshift.  These conditions are of different nature: algebraic, combinatorial, and related to the order in which the towers of level $n$ intersect the towers of level $n+1$. This allows retrieving results for primitive substitutive systems, linearly recurrent systems \cite{Cortez_Durand_Host_Maass_continuous_measurable_eigen_LR:2003}, among others.

\subsubsection{Conditions regarding towers heights }

We start by extending the result concerning \emph{exact finite rank} subshifts, i.e., $\cS$-adic subshifts such that $\cA = \cA_{\mu}$. The concept of exact finite rank subshift was introduced in \cite{Bezuglyi_Kwiatkowski_Medynets_Solomyak_Finite_rank_Bratteli:2013}, where it was proved that these systems are not mixing. Subsequently, partial rigidity was established in \cite{Danilenko_finite_rank_rationalerg_partial_rigidity:2016} and  part of the proof of the next lemma shares the approach used in that paper. 
However, the conditions of \cref{alturas A mu} are also satisfied by some superlinear complexity and non-finite exact rank $\cS$-adic subshifts. 

\begin{lemma} \label{lemabases}
Let $\boldsymbol \sigma = (\sigma_n \colon \cA_{n+1}^* \to \cA_n^*)_{n \in \N}$ be a clean, primitive and recognizable directive sequence, with $\cA_n=\cA$ for all $n\in \N$, and $\mu$ an ergodic measure on $(X_{\boldsymbol \sigma}, S)$. If $\cA_{\mu} \neq \cA$ and for every letter $b \in \cA \backslash \cA_{\mu}$
\begin{equation*}
\frac{\mu(B_b^{(n)})}{\mu(B^{(n)})} \xrightarrow[n \rightarrow \infty]{} 0,
\end{equation*}
then $(X_{\boldsymbol \sigma},\cB,\mu,S)$ is partially rigid.
\end{lemma}
\begin{proof}
Let $n \in \N$ and $a_0  \in \cA$ be such that $\mu(B_{a_0}^{(n)}) = \max_{a \in  \cA } \mu(B_{a}^{(n)})$. We can take $a_1 \in \cA$ such that 
\begin{equation*}
\mu(B_{a_1}^{(n)} \cap S^{h_{a_0}^{(n)}} B_{a_0}^{(n)}) \geq \frac{\mu(B_{a_0}^{(n)})}{d},
\end{equation*}
where $d = |\cA|$. Similarly, we can take $a_2 \in \cA$ such that
\begin{equation*}
\mu(B_{a_2}^{(n)} \cap S^{h_{a_1}^{(n)}} (B_{a_1}^{(n)} \cap S^{h_{a_0}^{(n)}} B_{a_0}^{(n)})) \geq \frac{\mu(B_{a_0}^{(n)})}{d^2}.
\end{equation*}
We can iterate this process to form a word $w = a_0 \ldots a_d \in \cA^*$ such that $\mu(B_w^{(n)}) \geq \frac{\mu(B_{a_0}^{(n)})}{d^d}$. Since the length of $w$ is $d+1$, there are two integers $0 \leq i < j \leq d$ such that $a_i = a_j$. We will refer to this letter as $a$ and we set $u = a_i a_{i+1} \ldots a_j$. Therefore, since
$S^{h_{a_0}^{(n)}+\ldots+h_{a_i-1}^{(n)}}B_w^{(n)}\subseteq B_u^{(n)} $, we have
\begin{equation*}
\mu(B_u^{(n)}) \geq 
\mu(S^{h_{a_0}^{(n)}+\ldots+h_{a_i-1}^{(n)}}B_w^{(n)})
= \mu(B_w^{(n)}) 
\geq \frac{\mu(B_{a_0}^{(n)})}{d^d} \geq \frac{\mu(B_{a}^{(n)})}{d^d}.
\end{equation*}
Multiplying by $h_a^{(n)}$ we get
\begin{equation} \label{inequality for height}
\mu(\sT_{u}^{(n)}) \geq \frac{\mu(\sT_{a}^{(n)})}{d^d}. 
\end{equation}

Now, assume that for all $n$ in an infinite subset $\cN \subseteq \N$ letter $a_0$ such that $\mu(B_{a_0}^{(n)}) = \max_{a \in  \cA } \mu(B_{a}^{(n)})$,
the word $u$ and the letter $a$ that fulfill (\ref{inequality for height}) are the same.
\bigskip

\noindent \underline{Claim:} $a_0, a \in \cA_{\mu}$.

Indeed, if $a_0 \in \cA \backslash \cA_{\mu}$, then
\begin{equation} \label{eq contradiccion}
1 = \frac{\mu(B^{(n)})}{\mu(B^{(n)})} = \frac{\sum_{b \in \cA} \mu(B_b^{(n)})}{\mu(B^{(n)})} \leq d \cdot \frac{\mu(B_{a_0}(n))}{\mu(B^{(n)})}  \xrightarrow[n \rightarrow \infty, n\in \cN]{} 0,
\end{equation}
which is a contradiction. Now, 
if $a \in \cA \backslash \cA_{\mu}$, then
\begin{equation*}
\frac{1}{d^d}\frac{\mu(B_{a_0}^{(n)})}{\mu(B^{(n)})} \leq \frac{\mu(B_u^{(n)}) }{\mu(B^{(n)}) } \leq \frac{\mu(B_a^{(n)})}{\mu(B^{(n)})} \xrightarrow[n \rightarrow \infty, n\in \cN]{} 0,
\end{equation*}
which produce the same contradiction given by inequalities in (\ref{eq contradiccion}).  

To conclude, we have  that for n large enough
\begin{equation*}
\frac{\eta}{d^d}  \leq  \frac{\mu(\sT_a^{(n)})}{d^d} \leq \mu(\sT_{u}^{(n)}),  
\end{equation*}
where $\eta>0$ is the constant stating whether the directive sequence is clean. The conclusion follows from \cref{cor KR}.
\end{proof}

\begin{theorem}
\label{alturas A mu}
Let $\boldsymbol \sigma = (\sigma_n : \cA_{n+1}^* \to \cA_n^*)_{n \in \N}$ be a clean, primitive and recognizable directive sequence, with $\cA_n=\cA$ for all $n\in \N$, and $\mu$ an ergodic measure on $(X_{\boldsymbol \sigma}, S)$.
If one of the following hypotheses holds:
\begin{enumerate}
    \item \label{hip1} $\cA_{\mu} = \cA$;
    \item \label{hip2} there is  $a \in \cA_{\mu}$ such that $\displaystyle\limsup_{n\to \infty}\frac{h^{(n)}_{a}}{h^{(n)}_{b}} < \infty $ for every $ b \in \cA \backslash \cA_{\mu}$,
\end{enumerate}
then $(X_{\boldsymbol \sigma},\cB,\mu,S)$ is partially rigid.   
\end{theorem}

\begin{proof}
For the first assumption, it is sufficient to repeat the proof of \cref{lemabases} and observe that we only need to prove that $a \in \cA_{\mu}$. But, since $\cA_{\mu} = \cA$, this is straightforward. 

For the second one, we will prove that the hypotheses of \cref{lemabases} are fulfilled, i.e., for all $b  \in \cA \backslash \cA_{\mu}$, $\frac{\mu(B_b^{(n)})}{\mu(B^{(n)})} \xrightarrow[n \rightarrow \infty]{} 0$.
Indeed, as $\displaystyle\limsup_{n\to \infty}\frac{h^{(n)}_{a}}{h^{(n)}_b} < \infty $ for every $ b \in \cA \backslash \cA_{\mu}$, we can fix $c \in (0,\infty)$ such that for infinitely many $n \in \N$ and for every $ b \in \cA \backslash \cA_{\mu}$, $h^{(n)}_a \leq c h^{(n)}_b$. Then, for an arbitrary $ b \in \cA \backslash \cA_{\mu}$
\begin{align*}
\limsup_{n \to \infty} \frac{\mu(B_b^{(n)}) }{\mu(B^{(n)}) } 
&\leq \limsup_{n \to \infty} \frac{\mu(B_b^{(n)}) }{\mu(B_a^{(n)}) } \\
&\leq c \limsup_{n \to \infty} \left( \frac{h^{(n)}_b}{h^{(n)}_a } \frac{\mu(B_b^{(n)}) }{\mu(B_a^{(n)}) } \right)  \\
&= c \limsup_{n \to \infty} \frac{\mu(\sT_b^{(n)}) }{\mu(\sT_a^{(n)}) } \\
&\leq \frac{c}{\eta} \limsup_{n \to \infty} \mu(\sT_b^{(n)})= 0, 
\end{align*}
where $\eta > 0$ is the constant defining whether the directive sequence is clean. Finally, we conclude by \cref{lemabases}.
\end{proof}

\begin{example}
    \cref{alturas A mu} allows to ensure the partial rigidity for systems that are not of exact finite rank. 
    Fix $\cA_n = \{0,1\}$ for all $n \geq 0$ and recall that for $\sigma_n : \cA_{n+1}^* \to \cA_n^*$, $(M_{\sigma_n})_{b,a} = | \sigma_n(a)|_b$ where $b \in \cA_n^*$ and $a\in \cA_{n+1}^*$. Let $\boldsymbol \sigma $ be a directive sequence such that for all $n \geq 1$ the incidence matrix of $\sigma_n$ is of the form
    \begin{equation*}
        M_{\sigma_n}  =  \left(
    \begin{array}{cc}
       a_n  & b_n  \\
        c  & d  
    \end{array}
    \right), 
    \end{equation*}
where $c,d$ are two positive integers and $(a_n)_{n \geq 0}$, $(b_n)_{n \geq 0}$ are integer sequences such that $\displaystyle\lim_{n \to \infty} a_n = \infty $ and there is a constant $K>0$, $K b_n \geq a_n$ for all $n \geq 1$. Then $(X_{\boldsymbol{\sigma}}, S)$ is uniquely ergodic with invariant measure $\mu$ satisfying $\displaystyle \lim_{n \to \infty} \mu( \sT^{(n)}_0 ) = 1$ and $\displaystyle \lim_{n \to \infty} \mu( \sT^{(n)}_1 ) = 0$. That means $\cA_{\mu} = \{0\}$ and in particular it is not of exact finite rank. Also notice that
\begin{equation*}
\limsup_{n \to \infty} \frac{h_0^{(n+1)}}{h_1^{(n+1)}} = \limsup_{n \to \infty} \frac{a_n h_0^{(n)} +  c h_1^{(n)}}{b_n h_0^{(n)} + d h_1^{(n)}} \leq \limsup_{n \to \infty} K \frac{b_n h_0^{(n)} +  \frac{c}{K} h_1^{(n)}}{b_n h_0^{(n)} + d h_1^{(n)}} < \infty
\end{equation*}
and, by \cref{alturas A mu}, $(X_{\boldsymbol \sigma},\cB,\mu,S)$ is partially rigid.
\end{example}

\begin{corollary} \label{cor proporcionalidad torres}
Let $\boldsymbol \sigma = (\sigma_n \colon \cA_{n+1}^* \to \cA_n^*)_{n \in \N}$ be a clean, primitive and recognizable directive sequence, with $\cA_n=\cA$ for all $n\in \N$, and $\mu$ an ergodic measure on $(X_{\boldsymbol \sigma}, S)$.
If there is a constant $c>0$ such that $h^{(n)}_a \leq c h^{(n)}_b$ for every $a,b \in \cA$ and infinitely many $n \in \N$, then $(X_{\boldsymbol \sigma},\cB,\mu,S)$ is partially rigid.
\end{corollary}
\begin{proof}
By assumption, $\displaystyle\limsup_{n\to \infty} \frac{h^{(n)}_a}{h^{(n)}_b} \leq c$ for every $a, b \in \cA$, so one of the hypotheses of \cref{alturas A mu} is satisfied.
\end{proof}

\subsubsection{Conditions regarding the order} \label{sec m-consecutive} In what follows, we study partial rigidity of $\cS$-adic subshifts that we call \emph{$m$-consecutive}. Danilenko proved that consecutive Bratteli-Vershik systems with a technical, but crucial, extra assumption are partially rigid (we refer to \cite[Theorem 7.6]{Danilenko_finite_rank_rationalerg_partial_rigidity:2016}). The $m$-consecutive property is slightly more general than this family of systems and, as we shall see, is sufficiently rich to build interesting examples. 
In addition, we show that the lower bound for the partial rigidity rate given for general $m$-consecutive sequences can be improved if we know that the directive sequence is of constant length.

\begin{definition}
A morphism $\sigma\colon \cA^* \to \cB^*$ is $m$-consecutive, where $m\geq 2$ is an integer, if for every $a \in \cA$, $\sigma(a)=b_1^{k_1}b_2^{k_2} \ldots b_r^{k_r}$, where $k_1,\ldots, k_r \geq m$ and $b_1,\ldots,b_r \in \cB$.  In that case, we can assume that consecutive $b_i's$ are different. 

A directive sequence $\boldsymbol \sigma$ is $m$-consecutive if there is $n_0 \in \N$ such that $\sigma_n$ is $m$-consecutive for every $n \geq n_0$.
\end{definition}

\begin{remark} \label{remark m-consective contraction}
If $\tau \colon \cA^* \to \cB^*$ is a morphism and $\sigma\colon \cB^* \to \cC^*$ a $m$-consecutive morphism, then $\sigma \circ \tau$ is a $m$-consecutive morphism.
Therefore, every contraction of a $m$-consecutive directive sequence is also $m$-consecutive. 
\end{remark}

\begin{theorem} \label{prop m consecutivo}
Let $m\geq 2$ and $\boldsymbol \sigma = (\sigma_n \colon \cA_{n+1}^* \to \cA_n^*)_{n \in \N}$ be a primitive, recognizable and $m$-consecutive directive sequence with finite alphabet rank $d$. Then, every ergodic measure for $(X_{\boldsymbol \sigma}, S) $ is $\left(\frac{m-1}{m} \cdot \frac{1}{d}\right)$-rigid. 
\end{theorem}

In the following proof, for a letter $b$, we say that $b^k$ is a $b$-block on $w$ if one of the following holds: $w_{[1,k]} = b^k$ and $w_{k+1} \neq b$; or $w_{(|w|-k,|w|]}= b^k$ and $w_{|w|-k} \neq$ b; or there exists $i \in \{2,\ldots,|w|-k\}$ such that $w_{[i,i+k)} = b^k$, $w_{i-1}\neq b$ and $w_{i+k} \neq b$.

\begin{proof}
Fix $\mu$ an ergodic measure and $n \geq n_0$, where $n_0$ is taken as in the definition of a $m$-consecutive directive sequence. Furthermore, since the alphabet rank of $\boldsymbol \sigma$ is $d$, except for contraction, as $m$-consecutivity is preserved, we can assume that $|\cA_n| = d$ for every $n \geq n_0$.

Consider $a \in \cA_{n+1}$ and $b \in \cA_n$ such that $b$ appears in $\sigma_n(a)$. 
Observe that $|\sigma_n(a)|_{bb} = \sum_{i=1}^{r} (k_i-1)$, where $k_i$ is the length of the $i$th $b$-block in $\sigma_n(a)$ in the order of occurrence and $r$ is the number of $b$-blocks in $\sigma_n(a)$. By $m$-consecutivity, $k_i \geq m$. So, 
\begin{align*}
\frac{|\sigma_n(a)|_{b} - |\sigma_n(a)|_{bb}}{|\sigma_n(a)|_{b}} = \frac{r}{\sum_{i=1}^{r}k_i} \leq \frac{r}{rm} = \frac{1}{m}.
\end{align*}
Therefore, by \cref{lemma torres w2},
\begin{equation*}
\mu(\sT_{bb}^{(n)} \cap \sT^{(n+1)}_a) \geq \frac{|\sigma_n(a)|_{bb}}{|\sigma_n(a)|_{b}} \mu(\sT_{b}^{(n)} \cap \sT^{(n+1)}_a) \geq \frac{m-1}{m} \mu(\sT_{b}^{(n)} \cap \sT^{(n+1)}_a).
\end{equation*}

Thus, summing over $a \in \cA_{n+1}$, we conclude that $\mu(\sT_{bb}^{(n)}) \geq \frac{m-1}{m} \mu(\sT_{b}^{(n)})$ (note that this conclusion only requires the $m$-consecutive condition, and the finite alphabet range condition is not necessary). 
Since $\sum_{b \in \cA_n} \mu(\sT^{(n)}_b) =1 $, there is at least one letter $b \in \cA_n$ such that $\mu(\sT^{(n)}_b) \geq \frac{1}{d}$.

In summary, for every $n \geq n_0$ there is a letter $b \in \cA_n$ such that $\mu(\sT_{bb}^{(n)}) \geq \frac{m-1}{m} \cdot \frac{1}{d}$ and we conclude by \cref{cor KR}.
\end{proof}

\begin{example}
As mentioned before, $m$-consecutive $\cS$-adic subshifts are a useful family for constructing examples of partially rigid systems with a variety of desired properties. In particular, by employing a construction similar to the one in Example 6.4 of \cite{Donoso_Durand_Maass_Petite_interplay_finite_rank_Sadic:2021} (see also \cite[Section 4]{Donoso_Durand_Maass_Petite_automorphism_low_complexity:2016}), we can create a partially rigid finite rank $\cS$-adic subshift with superlinear complexity. Note that to achieve superlinear complexity we necessarily need to build a system which is not a substitution or linearly recurrent subshift.

Let $(a_n)_{n \in \N}$ be an increasing sequence of integers larger than $4$ and let $\boldsymbol \sigma = (\sigma_n\colon  \{0,1\}^* \to \{0,1\}^*)_{n \in \N}$ be a directive sequence defined by
\begin{equation*}
\sigma_n(0) = 00111 \quad \text{ and } \quad \sigma_n(1) = 0^3 11 0^4 11 \cdots 11 0^{a_n+1} 11.
\end{equation*}
Let $(X_{\boldsymbol{\sigma}},S)$ be the $\cS$-adic subshift generated by $\boldsymbol \sigma$. It is not complicated to see that this directive sequence is recognizable, the proof is left to the reader. 

To prove that $\displaystyle \liminf_{n\to \infty} \frac{p_{X_{\boldsymbol \sigma}} (n)}{n} = \infty$ we show the stronger condition $\displaystyle \lim_{n\to \infty} p_{X_{\boldsymbol \sigma}} (n+1) - p_{X_{\boldsymbol \sigma}} (n)= \infty$. Notice that for every $n \geq 1$, $ p_{X_{\boldsymbol \sigma}} (n+1) - p_{X_{\boldsymbol \sigma}} (n) = r(n)$, where $r(n)$ is the number of right special words of length $n$, i.e., 
$r(n) = |\{w \in \cL_n(X_{\boldsymbol \sigma}): w0,w1 \in \cL_{n+1}(X_{\boldsymbol \sigma})\}|$. 
\medskip

\noindent \underline{Claim 1:} Fix $n \geq 1$ and $k \in \{4,\ldots, a_{n+1}\}$. Then, the word 
\begin{equation*}
W_n(k) = \sigma_{[0,n]}(0^{k-1}110^k) \sigma_{[0,n-1]}(00) \ldots \sigma_0(00)00,
\end{equation*}
is right special in $X_{\boldsymbol \sigma}$. 
\begin{proof}
Indeed, $0^{k-1}110^k0$ and $0^{k-1}110^k1$ belong to $\cL(X_{\boldsymbol \sigma}^{(n+1)})$,
so $\sigma_{[0,n]}(0^{k-1}110^k0)$ and $\sigma_{[0,n]}(0^{k-1}110^k1)$ belong to $\cL(X_{\boldsymbol \sigma})$. 
But a simple computation yields to
$$
\sigma_{[0,n]}(0^{k-1}110^k0)=\sigma_{[0,n]}(0^{k-1}110^k)
\sigma_{[0,n-1]}(00) \ldots \sigma_0(00)00 a u 
$$
and 
$$
\sigma_{[0,n]}(0^{k-1}110^k1)=\sigma_{[0,n]}(0^{k-1}110^k)
\sigma_{[0,n-1]}(00) \ldots \sigma_0(00)00 b v ,
$$
where $a, b \in \{0,1\}$ are different and $u,v \in \{0,1\}^*$.
Therefore, 
$W_n(k)0, W_n(k)1 \in \cL(X_{\boldsymbol \sigma})$, which allows us to conclude.
\end{proof}

For $n\in \N$, set $A_n = \lvert \sigma_{[0,n]}(1) \rvert$, $B_n = \lvert \sigma_{[0,n]}(0) \rvert$
and 
$C_{n}=|\sigma_{[0,n-1]}(00) \ldots \sigma_0(00)00| = 2(B_{n-1} + \ldots + B_0 + 1 )$, where $C_0=2$ by convention.
A simple but tedious computation yields
$\lfloor A_n/B_n \rfloor \geq c \ a_n$ for some positive constant $c$. Thus, by adjusting the values of the $a_n$'s we can assume that 
$\lfloor A_n/B_n \rfloor \geq 4$. \\

\noindent \underline{Claim 2:} Given $n \geq 1$, $k \in \{4,\ldots, a_{n+1}\}$ and $N \in \{kB_n+ C_n +1,\ldots, (k+1)B_n +C_n \}$ there are at least $\min(k, \lfloor A_n/B_n \rfloor)$ different right special words in $\cL_N(X_{\boldsymbol \sigma})$.
\begin{proof}
First, a simple computation shows that 
$$
|W_n(i)|=C_n+2A_n+(2i-1)B_n \geq  C_n  + (k+1)B_n  \geq N
$$
for each $i \in \{ 4+ k - \min(k,\lfloor A_n/B_n\rfloor),\ldots, k \}$.
Let $w_i$ be the suffix of length $N$ of $W_n(i)$. Then,
\begin{equation*}
w_i = p_i  \sigma_{[0,n]}(0^i) \sigma_{[0,n-1]}(00) \ldots \sigma_1(00)00,
\end{equation*}
where $p_i$ is a suffix of $\sigma_{[0,n]}(0^{i-1}11)$. From Claim 1 these words are right special. We are left to prove that they are distinct. Consider  $i, j\in \{ 4+ k - \min(k,\lfloor A_n/B_n\rfloor),\ldots, k \}$ with $i<j$ and notice that 
\begin{align*}
p_j  \sigma_{[0,n]}(0^{j-i}) \  \sigma_{[0,n]}(0^i) \sigma_{[0,n-1]}(00) \ldots \sigma_1(00)00 &= w_j \qquad  \text{ and }\\ 
p_i \ \sigma_{[0,n]}(0^i) \sigma_{[0,n-1]}(00) \ldots \sigma_1(00)00 &= w_i,
\end{align*}
then $w_i \neq w_j$ if and only if $p_i \neq p_j \sigma_{[0,n]}(0^{j-i})$. 

As $w_i$ and $w_j$ have the same length, $p_i$ and $p_j \sigma_{[0,n]}(0^{j-i})$ also have the same length, which is at least $|\sigma_{[0,n]}(0)|$. However, as $\sigma_{[0,n]}(0)$ is not a suffix of $\sigma_{[0,n]}(1)$, it cannot be a suffix of $p_i$. But $\sigma_{[0,n]}(0)$ is clearly a suffix of $p_j\sigma_{[0,n]}(0^{j-i})$. Thus, $p_i \neq p_j\sigma_{[0,n]}(0^{j-i})$ and we conclude. 

\end{proof}

\noindent \underline{Claim 3:} If the integer sequence $(a_n)_{n \in \N}$ satisfies $(a_{n+1}+1) > (5n + 7)A_n/B_n$, then $\displaystyle \lim_{n\to \infty} p_{X_{\boldsymbol \sigma}} (n+1) - p_{X_{\boldsymbol \sigma}} (n)= \lim_{n\to \infty} r(n) = \infty$
\begin{proof}
From Claim 2 we know that for every $n \geq 1$, $k \in \{4,\ldots,a_{n+1}\}$ and $N \in \{ kB_n  + C_n +1, \ldots,  (k+1)B_n + C_n \}$, $r(N) \geq \min(k, \lfloor A_n/B_n \rfloor )$. Also, since
$A_n/B_n \geq c \ a_n$, $\displaystyle \lim_{n \to \infty} \min(n,\lfloor A_n/B_n \rfloor ) = \infty $. 

Now, we have that 
\begin{align*}
\{m,m+1,\ldots \} &= \bigcup_{n\geq 4}\bigcup_{k=n}^{a_{n+1}}
\{kB_n  + C_n +1, \ldots, (k+1)B_n + C_n\} \\
&=\bigcup_{n\geq 4} 
\{nB_n  + C_n +1, \ldots, (a_{n+1}+1)B_n + C_n\}
\end{align*}
for some integer $m$ (which can be computed explicitly). Indeed, as $B_n \leq A_n$,
\begin{align*}
(n+1)B_{n+1} + C_{n+1} +1 &= (n+1)(2B_n + 3A_n) + 2B_n + C_n +1, \\
&\leq (5n+7)A_n + C_n  + 1 \leq  (a_{n+1}+1) B_n + C_n.
\end{align*}
Therefore, for any $N \geq m$ there exist $n\geq 4$, $k\in \{n,\ldots,a_{n+1}\}$ such that $N\in \{nB_n  + C_n +1, \ldots, (a_{n+1}+1)B_n + C_n\}$ and thus 
$$r(N)\geq \min(k,\lfloor A_n/B_n \rfloor ) \geq \min(n,\lfloor A_n/B_n \rfloor ),$$
and we conclude. 
\end{proof}

In conclusion, $(X_{\boldsymbol \sigma},S)$ has superlinear complexity and $\boldsymbol \sigma$ is $2$-consecutive, so for $\mu \in \cE(X_{\boldsymbol \sigma},S)$, by \cref{prop m consecutivo}, $(X_{\boldsymbol \sigma},\cB, \mu, S)$ is $\frac{1}{4}$-rigid. 
\end{example}

When specializing \cref{prop m consecutivo} to constant length $m$-consecutive directive sequences, we get the following theorem that does not require the hypothesis of finite topological rank.

\begin{theorem} \label{prop m consecutive Toeplitz}
Let $m\geq 2$ and $\boldsymbol \sigma = (\sigma_n \colon \cA_{n+1}^* \to \cA_n^*)_{n \in \N}$ be a primitive, recognizable, constant length and $m$-consecutive directive sequence. Then, every invariant measure for $(X_{\boldsymbol \sigma}, S) $ is $\frac{m-1}{m}$-rigid. 
\end{theorem}

\begin{proof}
We can use the first part of the proof of \cref{prop m consecutivo} which does not require the system to be of finite alphabet rank.
Thus, let $\mu$ be an ergodic measure and take $n \geq n_0$, then for every $b \in \cA_n$ we have
 \begin{equation*}
 \mu(\sT_{bb}^{(n)}) \geq \frac{m-1}{m} \mu(\sT_{b}^{(n)}).
 \end{equation*}
However, as the directive sequence has constant length, for every $a,b \in \cA_n$, $aa \sim_n bb$. Therefore, 
\begin{equation*}
\mu(\sT_{[aa]_n}^{(n)}) = \sum_{b \in \cA_n} \mu(\sT_{bb}^{(n)}) \geq \frac{m-1}{m} \sum_{b \in \cA_n} \mu(\sT_{b}^{(n)}) = \frac{m-1}{m}
\end{equation*}
and, as $\mu$ is an arbitrary ergodic measure, by \cref{theorem KR} we conclude that every ergodic measure is $\frac{m-1}{m}$-rigid (recall that in this context ergodic measures are non-atomic). Moreover, from \cref{remark partial seq} it is known that the partial rigidity sequence is $(h^{(n)})_{n \in \N}$, which in this case is independent of the measure, so by \cref{remark partially rigid invariant measures} every invariant measure is $\frac{m-1}{m}$-rigid.
\end{proof}

\begin{remark}
It is interesting to note that for every Toeplitz subshift $(X, S)$ (see \cite{Downarowicz_Toeplitz_survey:2005}  for a precise definition) there exists a minimal Cantor system that is rigid for every ergodic measure $\mu$ and is strong orbit equivalent to $(X,S)$. A simple proof can be constructed with the results of \cite{Gjerde_Johansen_Bratteli_Toeplitz:2000} and the elements developed in this section. 
\end{remark}

\begin{remark}
    In \cite[Theorem 1.4]{Cecchi_Donoso_SOE_superlinear_complexity:2022}, there is a construction of a recognizable, infinite alphabet rank, primitive and constant length directive sequence $\boldsymbol \sigma$ such that $(X_{\boldsymbol{\sigma}},S)$ has superlinear complexity and with infinitely many ergodic measures. A slight modification of that example allows us to ensure that the directive sequence is also $2$-consecutive and therefore, by \cref{prop m consecutive Toeplitz}, the system is partially rigid for every ergodic measure. Thus, the above criteria can be applied beyond non-superlinear complexity systems and exact finite rank systems. 
\end{remark}

\subsubsection{Conditions regarding return words} 
This condition may be one of the most restrictive. Still, it can be used to prove the partial rigidity for the natural coding of interval exchange transformations and, more generally, for Dendric subshifts (see for instance \cite{Berthe_Dolce_Durand_Leroy_Perrin_Rigidity_and_substitutive:2018, Durand_Perrin_Dimension_groups_dynamical_systems:2022}). Although the above are classical examples where the proposition can be used, the range of applications is broader. This is because the condition only requires finiteness of the return words sets for infinitely many levels and not for every level.

\begin{proposition} \label{teoretorno2}
 Let $\boldsymbol \sigma = (\sigma_n \colon \cA_{n+1}^* \to \cA_n^*)_{n \in \N}$  be a primitive and recognizable directive sequence. If there is a constant $c > 0$ such that $ \sum_{a \in \cA_n} | \cR_{X_{\boldsymbol \sigma}^{(n)}}(a)| \leq c $ for infinitely many $n \in \N$,  
 then $(X_{\boldsymbol \sigma},\cB,\mu,S)$ is $\frac{1}{c}$-rigid for every ergodic measure $\mu$.
\end{proposition}
\begin{proof}
Observe that
\begin{equation} 
\sT^{(n)}_a = \bigcup_{w \in \cR_{X_{\boldsymbol \sigma}^{(n)}}(a)} \sT_{aw}^{(n)}
\end{equation}
and thus
\begin{equation} 
X_{\boldsymbol \sigma} = \bigcup_{a \in \cA_n} \bigcup_{w \in \cR_{X_{\boldsymbol \sigma}^{(n)}}(a)} \sT_{aw}^{(n)}.
\end{equation}
Let $\cN\subseteq \N$ be infinite such that $ \sum_{a \in \cA_n} | \cR_{X_{\boldsymbol \sigma}^{(n)}}(a)| \leq c $ for all $n\in \cN$. Then, we have that there exists a sequence of complete words $(w(n))_{n \in \cN}$ with $w(n) = a_nw'(n)$, where $w'(n) \in \cR_{X_{\boldsymbol \sigma}^{(n)}}(a_n)$,  such that $\mu(\sT^{(n)}_{w(n)}) \geq \frac{1}{c} $. From this we conclude using \cref{cor KR}.
\end{proof}

We deduce the following.

\begin{corollary} \label{cor finite return}
Let $(X,S)$ be a minimal subshift on an alphabet $\cA$. If there is a number $d>0$ such that $|\cR_X (w) | \leq d$ for every word $w \in \cL(X)$, then  $(X,\cB,\mu,S)$ is $\frac{1}{d^2}$-rigid for every ergodic measure $\mu$.
\end{corollary} 

Before proving the corollary, we state two lemmas whose proofs can be found in \cite{Durand_Perrin_Dimension_groups_dynamical_systems:2022}. To state them, note that given a word $u \in \cL(X)$ appearing in a minimal subshift $X$, the set $U = \cR_X(u)$ is a \emph{circular code}, i.e., every word $w \in \cA^*$ admits at most one decomposition as a concatenation of elements of $U$ and if $uv, vu \in U^*$, then $u,v \in U^*$. A \emph{circular morphism} is a morphism $\varphi^* \colon \cB^* \to \cA^*$ whose restriction to $\cB$ is a bijection into a circular code $U \subseteq \cA^*$.

\begin{lemma}{\cite[Proposition 1.4.32]{Durand_Perrin_Dimension_groups_dynamical_systems:2022}}
\label{lemma circular morphism}   
A (non-erasing) morphism $\sigma\colon \cB^* \to \cA^*$ is recognizable on $\cB^{\Z}$ if and only if $\sigma$ is a circular morphism.
\end{lemma}

\begin{lemma}{\cite[Lemma 6.4.11]{Durand_Perrin_Dimension_groups_dynamical_systems:2022}} \label{lemma return construction}
Let $(X,S)$ be a minimal subshift on an alphabet $\cA$ and $(u_n)_{n \in \N}$ be a sequence of words in $\cL(X)$, with $u_0 = \varepsilon$, such that $u_n$ is a proper suffix of $u_{n+1}$ for every $n\in \N$. Let $\cA_n = \{ 1,\ldots, |\cR_X(u_n)| \}$ and $\alpha_n \colon \cA_n^* \to \cA^*$ a coding morphism for $\cR_X(u_n)$. Let $\boldsymbol \tau $  be the sequence of morphisms such that $\alpha_n \circ \tau_n = \alpha_{n+1}$ for every $n \in \N$. Then $\boldsymbol \tau$ is a primitive directive sequence such that $X_{\boldsymbol \tau} = X$.
\end{lemma} 
\begin{proof}[Proof of \cref{cor finite return}]
Let $(u_n)_{n \in \N}$ a sequence of words in $\cL(X)$ such that $u_0 = \varepsilon$, $u_{n+1} \in \cR_X(u_n)$ and $|u_{n+1}|>|u_n|$.  In particular, $u_n$ is a proper suffix of $u_{n+1}$. Let $\boldsymbol \tau = (\tau_n \colon \cA_{n+1 }^* \to \cA_n^* )_{n \in \N}$ be the directive sequence associated with $(u_n)_{n \in \N}$ as in \cref{lemma return construction}. 

By construction of $\boldsymbol \tau$, $\tau_n(\cA_{n+1}) = \cR_{X_{\boldsymbol{\tau}}^{(n)}}(a_n)$, where $a_n \in \cA_n$ is the letter such that $\tau_{[0,n)}(a) = u_{n+1}$. Thus, by \cref{lemma circular morphism}, $\tau_n$ is recognizable for every $n \in \N$ and, therefore, $\boldsymbol \tau$ is recognizable.

By assumption, for every $n \in \N$, $|\cR_X(u_n)| \leq d$, and then $|\cA_n| \leq d $. Moreover, for every $n\in \N$ and $a \in \cA_n$, $\tau_{[0,n)}(\cR_{X^{(n)}_{\boldsymbol \tau}} (a)) =  \cR_X (\tau_{[0,n)}(a))$, and therefore $|\cR_{X^{(n)}_{\boldsymbol \tau}} (a)| \leq d$. Therefore, $\sum_{a \in \cA_n} | \cR_{X^{(n)}_{\boldsymbol \tau}} (a) | \leq d^2$. Applying \cref{teoretorno2} to the directive sequence 
$\boldsymbol{\tau}$, we obtain that $X_{\boldsymbol{\tau}}$ is $\frac{1}{d^2}$-rigid for every ergodic measure.
\end{proof}

\begin{remark}
Every statement involving the set of right return words $\cR_X(w)$ (\cref{teoretorno2}, \cref{cor finite return} and \cref{lemma return construction}) also applies to the set of left return words $\cR_X'(w)$.  
\end{remark}

\subsubsection{Condition regarding the directive sequence}
The repetition of a positive morphism $\sigma$ infinitely many times in a recognizable directive sequence $\boldsymbol \sigma$ implies the unique ergodicity of $(X_{\boldsymbol \sigma},S)$ (see \cite{Veech_interval_exchange:1978} or \cite{Bezuglyi_Kwiatkowski_Medynets_Solomyak_Finite_rank_Bratteli:2013}). This condition applies to substitution subshifts and, more generally, to linearly recurrent  subshifts. It is not clear whether this condition implies non-superlinear complexity or not.
In this section, we will see that this condition also implies partial rigidity. 

For a morphism $\sigma\colon \cA^* \to \cB^*$, denote $\lVert \sigma \rVert = \max_{a \in \cA} |\sigma(a)|$. 

\begin{lemma} \label{lemma boundary}
Let $\boldsymbol \sigma = (\sigma_n \colon \cA_{n+1}^* \to \cA_n^*)_{n \in \N}$  be a primitive and recognizable directive sequence. For an integer $n\geq 1$ and a letter $b \in \cA_n$, if $|\sigma_n(a)|_b >0$ for every $a \in \cA_{n+1}$, then $ |w| \leq 2 \lVert \sigma_n \rVert-1$ for every $w \in \cR_{X^{(n)}_{\boldsymbol \sigma}}(b) $.
\end{lemma}
\begin{proof} If $w \in \cR_{X^{(n)}_{\boldsymbol \sigma}}(b)$, then $|w|_b = 1$ and 
 $bw \in \cL^{(n)} (\boldsymbol \sigma)$. 
Thus, there is a word $u \in \cL^{(n+1)} (\boldsymbol \sigma)$ such that $bw \sqsubseteq \sigma_n(u)$. Since $|\sigma_n(a)|_b >0$ for every $a \in \cA_{n+1}$, there exist two letters $a_1$ $a_2\in \cA_{n+1}$ such that $bw \sqsubseteq \sigma_n(a_1a_2)$.
Then, $|bw|  \leq |\sigma_n(a_1)| + |\sigma_n(a_2)| \leq 2 \lVert \sigma_n \rVert$, from which we conclude.
\end{proof}

\begin{theorem} \label{prop morfismos repetidos}
Let $\boldsymbol \sigma = (\sigma_n \colon \cA_{n+1}^* \to \cA_n^*)_{n \in \N}$  be a primitive and recognizable directive sequence such that there exists a positive morphism $\tau\colon\cA^*\to \cB^*$ which is repeated infinitely many times in $\boldsymbol \sigma$ 
(i.e., $|\cN| = |\{ n \in \N : \sigma_n = \tau \}| = \infty$). Then $(X_{\boldsymbol \sigma},\cB,\mu,S)$ is partially rigid, where $\mu$ is its unique invariant measure.
\end{theorem} 

\begin{proof}
Fix $n \in \cN$. So $\sigma_n = \tau$, $\cA_n=\cB$ and $\cA_{n+1}=\cA$. 
By hypothesis, $|\sigma_n(a)|_b >0$ for all $b \in \cA_n$ and $a \in \cA_{n+1}$. Then, \cref{lemma boundary} gives that $|w| \leq 2 \lVert \tau \rVert -1$
for every $b \in \cA_n$ and $w \in \cR_{X_{\boldsymbol \sigma}^{(n)}}(b)$. 
Therefore,  $|\cR_{X_{\boldsymbol \sigma}^{(n)}} (b)| \leq |\cA|^{2 \lVert \tau \rVert -1}$ for every $b \in \cB$ and 
$\sum_{b\in \cB} |\cR_{X_{\boldsymbol \sigma}^{(n)}} (b)| \leq |\cB|\cdot|\cA|^{2 \lVert \tau \rVert -1}$. We conclude by \cref{teoretorno2}.
\end{proof}

\begin{question} \label{question finite rank}
    In this section, we presented several sufficient conditions to ensure partial rigidity on $\cS$-adic subshifts, especially for finite alphabet rank subshifts. When this paper was first released on arXiv, we wondered whether a mixing finite alphabet rank system existed. B. Espinoza has communicated to us that he has found such an example, and he is currently working on a publication. It is still open to give a complete characterization, or even to have sufficient conditions, in order to determine whether a directive sequence $\boldsymbol \sigma$ defines an $\cS$-adic subshift that admits a mixing or a partially mixing measure. 
Similarly, we do not know if there are uniquely ergodic finite alphabet rank $\cS$-adic subshifts that are neither mixing nor partially rigid.
\end{question}

\section{Rigidity in $\cS$-adic subshifts} \label{sec:S-adic rigid}
In this section we focus on conditions under which $\cS$-adic subshifts are rigid (i.e., the partial rigidity rate equals 1). We provide necessary conditions for a subshift to be rigid, based on its language. As a consequence, we show that most words are complete (see \cref{teoLRrigid} and \cref{remark sequence rigidity p/q}). We also construct families of examples of subshifts, which are rigid thanks to the criteria developed in \cref{sec:partial_rigidity}. 

\subsection{Necessary conditions for rigidity}

We recall that for a subshift $X$, set of complete words on $\cL(X)$ is denoted by $\cC \cL (X)$. We start with a general lemma. 

\begin{lemma} \label{propnecesaria}
Let $\boldsymbol \sigma = (\sigma_n \colon \cA_{n+1}^* \to \cA_n^*)_{n \in \N}$ be a primitive and recognizable directive sequence and $\mu$ be an ergodic measure on $X_{\boldsymbol \sigma}$. If $(X_{\boldsymbol \sigma}, \cB, \mu,S)$ is rigid then,
\begin{equation}
\limsup_{n \rightarrow \infty} \left[ \mu \left( \bigcup_{w \in \cC\cL_n(X_{\boldsymbol \sigma})} [w]_{X_{\boldsymbol \sigma}} \right) \right] = 1,
\end{equation}
where $\cC\cL_n(X_{\boldsymbol \sigma}) = \cC\cL(X_{\boldsymbol \sigma}) \cap \cA_0^{n} $.
\end{lemma}

\begin{proof}
As  $(X_{\boldsymbol \sigma}, \cB, \mu,S)$ is rigid, by \cref{cor equivrigid} and \cref{theorem partial rigidity rate}, there is a sequence of complete words $(w(n))_{n\in \N}$, $w(n) \in \cC \cL( X^{(n)}_{\boldsymbol \sigma})$
such that
\begin{equation*} 
 \lim_{n \rightarrow \infty} \mu(\sT_{[w(n)]_n}^{(n)}) = 1.
\end{equation*}
Fix $n\in \N$, $u = u_1 \ldots u_{r} \in [w(n)]_n$ and $\ell = |\sigma_{[0,n)}(u_1) |$.
For $0 \leq k \leq \ell$, let $p_{k}^{u_1,n}$ be the $k$-th prefix of $\sigma_{[0,n)}(u_1)$ (i.e., 
$|p_{k}^{u_1,n}| = k$ and it is a prefix of $\sigma_{[0,n)}(u_1)$) and $s_{k}^{u_1,n}$ be the $k$-th suffix of $\sigma_{[0,n)}(u_1)$ (i.e.,  $|s_{k}^{u_1,n}| = \ell - k$ and it is a suffix of  $\sigma_{[0,n)}(u_1)$). Notice that $p_{0}^{u_1,n}$ and $s_{\ell}^{u_1,n}$ are empty words. Then
\begin{equation*}
\sT_{u}^{(n)} \subseteq \bigcup_{v \in \cC_{u}^n} [v]_{X_{\boldsymbol \sigma}},
\end{equation*}
where $\cC_{u}^n = \{ s_{k}^{u_1,n} \sigma_{[0,n)}(u_2 \ldots u_{r-1}) p_{k+1}^{u_1,n} : 0 \leq k < \ell \} \subseteq \cA^*_0$. Moreover, every word $v \in \cC_u^n$ is a complete word because the last letter of $p_{k+1}^{u_1,n}$ is equal to the first letter of $s_{k}^{u_1,n}$. Thus, by definition
\begin{equation} \label{igualdadconj}
\sT_{[w(n)]_n }^{(n)} \subseteq \bigcup_{u \in [w(n)]_n} \bigcup_{v \in \cC_{u}^n} [v]_{X_{\boldsymbol \sigma}}.
\end{equation}

To conclude, notice that,  by the definition of $\sim_n$, for $m_n = |\sigma_{[0,n)} (u_1 \ldots u_{r-1})| + 1$ one has that if $ \displaystyle v \in \bigcup_{u \in [w(n)]_n} \cC_{u}^n$ then $|v| = m_n$. Therefore, for every $n \in \N$, $ \displaystyle \bigcup_{u \in [w(n)]_n} \cC_{u}^n  \subseteq \cC \cL_{m_n} ({X}) $ and then
\begin{align*}
    1 = \limsup_{n \rightarrow \infty} \mu (\sT_{[w(n)]_n}^{(n)}) &\leq \limsup_{n \rightarrow \infty}  \mu \left(   \bigcup_{u \in [w(n)]_n} \bigcup_{v \in \cC_{u}^n} [v]_{X_{\boldsymbol \sigma}} \right) \\
    & \leq \limsup_{n \rightarrow \infty} \mu \left( \bigcup_{v \in \cC\cL_{m_n}(X)} [v]_{X_{\boldsymbol \sigma}} \right) \\
    &\leq \limsup_{n \rightarrow \infty} \mu \left( \bigcup_{v \in \cC\cL_{n}(X)} [v]_{X_{\boldsymbol \sigma}} \right). \\
\end{align*}
This allows us to conclude.
\end{proof}

\begin{remark} \label{remark sequence rigidity}
Similarly to \cref{remark partial seq}, in the previous proof it can be observed that $(m_n -1)_{n \in \N}$ is a rigidity sequence for the system $(X_{\boldsymbol \sigma}, \cB, \mu, S)$.
\end{remark}

We recall that a linearly recurrent subshift is an $\cS$-adic subshift $(X_{\boldsymbol \sigma}, S)$ such that there exists a finite set of morphisms $\cS$ such that the directive sequence $\boldsymbol \sigma = (\sigma_n \colon \cA_{n+1}^* \to \cA_n^*)_{n \in \N}$ is proper, positive and $\sigma_n \in \cS$ for all $n \geq 1$. Linearly recurrent subshifts are uniquely ergodic. We will use that if $(X_{\boldsymbol \sigma}, S)$ is linearly recurrent then there is a constant $L>0$ such that $\mu([w]_{X_{\boldsymbol{\sigma}}}) \leq L \mu([u]_{X_{\boldsymbol \sigma}})$ for all words $u,w \in \cL(X_{\boldsymbol \sigma})$ of the same length (see \cite[Proposition 13]{Durand2000}).

\begin{theorem} \label{teoLRrigid} 
Let $(X_{\boldsymbol \sigma}, S)$ be a linearly recurrent subshift and let $\mu$ be its unique invariant measure. If $(X_{\boldsymbol \sigma}, \cB, \mu,S)$ is rigid, then
\begin{equation} \label{lim q/p}
\limsup_{n \rightarrow \infty} \frac{q_{X_{\boldsymbol \sigma}}(n)}{p_{X_{\boldsymbol \sigma}}(n)} = 1,
\end{equation}
where $p_{X_{\boldsymbol \sigma}}(n) = | \cL_n(X_{\boldsymbol \sigma})|$ and $q_{X_{\boldsymbol \sigma}}(n)= | \cC \cL_n(X_{\boldsymbol \sigma})|$. 
\end{theorem}
To prove this theorem, we use the following lemma.
\begin{lemma} \label{lemacomplex} 
 Let $(X_{\boldsymbol \sigma}, S)$ be a linearly recurrent subshift and let $\mu$ be its unique invariant measure. Let $(W_n)_{n\in \N}$ such that $W_n \subseteq \cL_n(X_{\boldsymbol \sigma})$ for all $n\in \N$. 
Then, for any increasing sequence of integers $(m_n)_{n\in \N}$ we have that
\begin{equation*}
\frac{|W_{m_n}|}{p_{X_{\boldsymbol \sigma}}(m_n)} \xrightarrow[n \rightarrow \infty]{} 0~~ \text{ if and only if} ~~
 \mu \left( \bigcup_{w \in W_{m_n}} [w]_{X_{\boldsymbol \sigma}} \right ) \xrightarrow[n \rightarrow \infty]{} 0.
\end{equation*}
\end{lemma}
\begin{proof}
For $m\in \N$, let $ w_m, u_m \in \cL_m(X_{\boldsymbol \sigma})$ such that $\mu( [w_m]_{X_{\boldsymbol \sigma}} ) = \max_{w \in \cL_m (X_{\boldsymbol \sigma})} \mu( [w]_{X_{\boldsymbol \sigma}} )$ and $\mu( [u_m]_{X_{\boldsymbol \sigma}} ) = \min_{w \in \cL_m (X_{\boldsymbol \sigma})} $ $ \mu( [w]_{X_{\boldsymbol \sigma}} )$. Let $L>0$ such that $\mu( [w_m]_{X_{\boldsymbol \sigma}} )\leq L\mu( [u_m]_{X_{\boldsymbol \sigma}} )$ for all $m\in \N$.
Combining this with the equality  \[\mu\left( \bigcup_{w \in W_m} [w]_{X_{\boldsymbol \sigma}} \right) = \frac{ \mu( \bigcup_{w \in W_m} [w]_{X_{\boldsymbol \sigma}} ) }{ \mu( \bigcup_{w \in \cL_m({X_{\boldsymbol \sigma}})} [w]_{X_{\boldsymbol \sigma}} ) }=  \frac{ \sum_{w \in W_m} \mu( [w]_{X_{\boldsymbol \sigma}} ) }{ \sum_{w \in \cL_m({X_{\boldsymbol \sigma}})} \mu(  [w]_{X_{\boldsymbol \sigma}} ) }  \]
we obtain that 
\[  \frac{1}{L}\frac{|W_m|}{p_{X_{\boldsymbol \sigma}}(m)}\leq  \mu\left( \bigcup_{w \in W_m} [w]_{X_{\boldsymbol \sigma}} \right) \leq L \frac{|W_m|}{p_{X_{\boldsymbol \sigma}}(m)}  \]
from where the conclusion follows. 
\end{proof}

\begin{proof}[Proof of \cref{teoLRrigid}]
By \cref{propnecesaria}, since $\mu$ is rigid we have that
\begin{equation*}
\lim_{n \rightarrow \infty} \mu \left( \bigcup_{w \in \cC\cL_{m_n}({X_{\boldsymbol \sigma}})} [w]_{X_{\boldsymbol \sigma}} \right) = 1,
\end{equation*}
for an increasing sequence of integers $(m_n)_{n \in \N}$.
Setting $W_n = \cL_{m_n}(X_{\boldsymbol \sigma}) \backslash \cC \cL_{m_n} (X_{\boldsymbol \sigma}) $, by \cref{lemacomplex}, 
\begin{equation*}
\lim_{n \to \infty} \frac{|W_n|}{p_{X_{\boldsymbol \sigma}}(m_n)} = 0. 
\end{equation*}
We conclude by noting that $|W_n| + q_{X_{\boldsymbol \sigma}}(m_n) = p_{X_{\boldsymbol \sigma}}(m_n)$ for every $n \in \N$. 
\end{proof}

\begin{remark} \label{remark sequence rigidity p/q}
As in \cref{remark sequence rigidity}, the sequence $(m_n)_{n \in \N}$ for which the limit in (\ref{lim q/p}) is reached was constructed so that $(m_n-1)_{n \in \N}$ is a rigidity sequence. Thus, \cref{teoLRrigid} implies that, under the same assumption, there is a rigidity sequence $(n_k)_{k \in \N}$ such that $\displaystyle \lim_{k \to \infty} \frac{q_{X_{\boldsymbol \sigma}} (n_k + 1)}{p_{X_{\boldsymbol \sigma}} (n_k + 1)} = 1$.
\end{remark}

\begin{remark}
It is worth mentioning that for aperiodic subshifts $ q_{X} (m) < p_{X}(m)$ for all $m\in \N$. Indeed, assume that $ \cC\cL_m(X) = \cL_{m}(X)$ for some $m$, and let $u = u_1  \ldots u_m $ be a word in $\cL_m(X)$ and $a \in \cA$ such that $ u_1 u_2 \ldots u_m a \in \cL(X)$. Then $u_2 u_3 \ldots u_m a \in \cL_m(X)$ and so $a = u_2$. Therefore, every word of length $m$ is uniquely extendable, and so $X$ is periodic.
\end{remark}

\subsection{A rigid $\cS$-adic subshift with multiple ergodic measures}  
In order to illustrate some applications of the results in \cref{section partial rigidity Sadic}, we construct an $\cS$-adic subshift with $r >0$ ergodic measures, each one rigid for the same rigidity sequence. Examples with the same property can be found for some Toeplitz systems \cite{Williams_toeplitz_non_ue:1984} To introduce the $\cS$-adic subshift, we quote the following lemma. We need the following notion: for a vector $v \in \R^{\cA}$, set $|v| = \sum_{a \in \cA} |v(a)|$. 

\begin{lemma}{\cite[Theorem 11]{Arbulu_Durand_Espinoza_jacobs-keane_Sadic:2023}} \label{lemma uniquely ergodic} Let $\boldsymbol \sigma = (\sigma_n \colon \cA^*_{n+1} \to \cA^*_n)_{n \in \N}$ be a primitive, recognizable and constant length directive sequence. The following are equivalent:
\begin{enumerate}
\item There exists a contraction $\boldsymbol {\sigma'} = (\sigma_k' = \sigma_{[m_k , m_{k+1})})_{k \in \N}$ of $\boldsymbol \sigma$ such that the vectors $(v_k)_{k\in \N}$ given by 
\begin{equation*}
v_k(a) = \min_{b \in \cA_{m_k +1} } | \sigma'_k (b)|_a, \quad a \in  \cA_{m_k }, \quad k \in \N
\end{equation*}
satisfy
\begin{equation*}
\sum_{k \in \N} \frac{|v_k|}{|\sigma'_k|} = \infty.
\end{equation*}
 \item The system $(X_{\boldsymbol \sigma},S)$ is uniquely ergodic. 
\end{enumerate}
\end{lemma}

\begin{proposition} \label{ex1 rigid}
    For every $r \geq 1$, there exists a minimal $\cS$-adic subshift $(X_{\boldsymbol \sigma}, S)$ such that $|\cE (X_{\boldsymbol \sigma}, S)| = r$ and every ergodic measure is rigid for the same rigidity sequence. 
\end{proposition}

\begin{proof}
Let $\cA = \{ a_1, b_1, a_2, b_2, \ldots, a_r, b_r\} $. We define $\boldsymbol \sigma = ( \sigma_n \colon \cA^* \to \cA^* )_{n \in \N} $ such that
\begin{align*}
\sigma_n(a_i) &= a_1 a_2 \ldots a_r (a_i b_i)^{2^n} b_1 b_2 \ldots b_r, \\
\sigma_n(b_i) &= a_1 a_2 \ldots a_r ( b_i a_i)^{2^n} b_1 b_2 \ldots b_r,
\end{align*}
for $i\in \{1,\ldots,r\}$, which is recognizable, positive and proper. Moreover, $\boldsymbol \sigma $ has constant length and we denote by $h^{(n)}$ the height at level $n$. \\

\noindent\underline{Claim:} There are $r$ ergodic measures $\mu_1, \ldots, \mu_r$ and for each $i \in \{1,\ldots,r\}$, $\cA_{\mu_i} = \{a_i,b_i\}$.

 Notice that, if $\mu$ is an ergodic measure, then it is clear that $a_i \in \cA_{\mu}$ if and only if $b_i \in \cA_{\mu}$, because 
 \begin{align*}
    \mu(\sT_{a_i}^{(n)}) &= \frac{h^{(n)}}{h^{(n+1)}} \left( 2^{n+1} (\mu(\sT_{a_i}^{(n+1)}) + \mu(\sT_{b_i}^{(n+1)})) + \sum_{j \neq i} (\mu(\sT_{a_j}^{(n+1)}) + \mu(\sT_{b_j}^{(n+1)})) \right) \\  &= \mu(\sT_{b_i}^{(n)}). 
 \end{align*}

To conclude the claim, note that there are two possibilities: either there is a unique invariant measure $\mu$ such that $\cA_{\mu} = \cA$ or there are $r$ ergodic measures $\mu_i$ such that $\cA_{\mu_i} = \{a_i, b_i\}$ for all $i \in \{1,\ldots,r\}$. This is trivial when $r \in \{1,2\} $. Suppose without loss of generality that $r\geq 3$, $a_1, a_2,b_1,b_2 \in \cA_{\mu}$ for some ergodic measure $\mu$ and define $\phi_i \colon X_{\boldsymbol\sigma} \to X_{\boldsymbol\sigma}$ as the homeomorphism that for every $x \in X_{\boldsymbol\sigma}$ interchanges the letters $a_1, b_1$ with $a_i, b_i$ respectively for $ i \in \{3,\ldots,r\}$.
Then $\phi_i$ commutes with $S$, and the pushforward measure $\mu' = \phi_i \mu $ is a $S$-ergodic probability measure such that $a_i, b_i , a_2, b_2 \in \cA_{\mu'}$. In particular, $\cA_{\mu} \cap \cA_{\mu'} \neq \emptyset$, and therefore $\mu = \mu'$ and $a_i, b_i \in \cA_{\mu}$. Since $i \in \{3,\ldots, r\}$ is arbitrary, we conclude that $\cA_{\mu} = \cA$. Finally, assuming that $\cA_{\mu_1} = \{a_1, b_1\}$, if we define $\mu_i = \phi_i \mu_1$ for every $i \in \{1,\ldots,r\}$, then $\cA_{\mu_i} = \{a_i,b_i\}$. 

Now, using the same notation as in \cref{lemma uniquely ergodic}, for every contraction $\boldsymbol{\sigma'}$ of $\boldsymbol{\sigma}$, $\frac{|v_k|}{|\sigma'_k|} = O(\frac{1}{2^k})$ for every $k$ and therefore $\displaystyle
\sum_{k \in \N} \frac{|v_k|}{|\sigma'_k|} < \infty$. Thus, by \cref{lemma uniquely ergodic}, the system cannot be uniquely ergodic and we conclude the claim. \\

\noindent\underline{Claim:} For $i \in \{1,\ldots,r\}$,  $\mu_i$ is rigid and $(2 h^{(n)})_{n \in \N} $ is a rigidity sequence.

\noindent Notice that, as $\cA_{\mu_i} = \{a_i,b_i\} $, $\mu_i((\sT^{(n)}_{a_i} \cup \sT^{(n)}_{b_i}) \cap (\sT^{(n+1)}_{a_i} \cup \sT^{(n+1)}_{b_i} )) \xrightarrow[n \to \infty]{} 1 $. Also, by \cref{lemma torres w},
\begin{align*}
\mu_i(\sT^{(n)}_{a_i b_i a_i}) 
&= \mu_i(\sT^{(n)}_{a_i b_i a_i} \cap \sT^{(n+1)}_{a_i}) + \mu_i(\sT^{(n)}_{a_i b_i a_i} \cap \sT^{(n+1)}_{b_i}) \\
&\geq \frac{n^2-1}{ n^2 + 1}  \mu_i(\sT^{(n)}_{a_i } \cap \sT^{(n+1)}_{a_i}) + \frac{n^2-1}{ n^2 + 1}  \mu_i(\sT^{(n)}_{a_i } \cap \sT^{(n+1)}_{b_i}),
\end{align*}
and the same applies to $ \mu(\sT^{(n)}_{b_i a_i b_i}) $. 
As $f_{\cA}(a_i b_i) = f_{\cA} (b_ia_i)$, by \cref{lemma abelianization}, $a_i b_i a_i \sim_n b_i a_i b_i$, so 
\begin{equation*}
\mu_i(\sT^{(n)}_{[a_i b_i a_i]_n}) \geq \frac{n^2-1}{n^2+1} \left( \sum_{j,l \in \{a_i,b_i\} } \mu_i(\sT^{(n)}_{j } \cap \sT^{(n+1)}_{l})  \right)
\xrightarrow[n \to \infty]{} 1.
\end{equation*}
Then $(X_{\boldsymbol \sigma}, \cB, \mu_i, S)$ is rigid for every $i \in \{1,\ldots,r\}$ and a rigidity sequence is given by $h^{(n)}_{a_i} + h^{(n)}_{b_i} = 2 h^{(n)}$. Finally, by \cref{remark partially rigid invariant measures}, every invariant measure is rigid. 
\end{proof}

\section{Partial rigidity rate and constant length morphisms} \label{sec p rigid constant}

In this section, we show the relationship between the partial rigidity rate and the induced system on the basis of an $\cS$-adic subshift with constant length directive sequence. Thanks to this, we find an upper bound of the rigidity rate for a family of substitution subshifts that allows us to compute the exact value in some cases. The most notable is the Thue-Morse subshift. 

\subsection{Partial rigidity rate in constant length $\cS$-adic subshifts}

\begin{theorem}  \label{theorem toeplitz delta mu}
Let $\boldsymbol \sigma = (\sigma_n \colon \cA_{n+1}^* \to \cA_n^*)_{n \in \N}$  be a recognizable, constant length and primitive directive sequence. Let $\mu$ be an $S$-invariant ergodic measure on $X_{\boldsymbol \sigma}$. Then
\begin{equation} \label{eq Toeplitz delta mu}
\delta_{\mu} = \inf_{n \geq 1}  \sup_{\ell \geq 2} \left\{ \sum_{\substack{w \in \cC\cL^{(n)}(\boldsymbol \sigma) \\ \lvert w \rvert  = \ell }}   \mu^{(n)}  ([w]_{X^{(n)}_{\boldsymbol \sigma}}) \right\},
\end{equation}
where $ \mu^{(n)}  $ is the induced measure on the base $B^{(n)}$ of the natural sequence of Kakutani-Rokhlin partitions $(\sP^{(n)})_{n \in \N}$ associated with $(X_{\boldsymbol \sigma},S)$.
\end{theorem}

\begin{proof}
By \cref{theorem partial rigidity rate}, 
\begin{equation*} 
\delta_{\mu} = \inf_{n \geq 1} \left\{ \sup_{\substack{w \in \cA_n^* \\ w_1 = w_{|w|} }}  \mu ( \sT_{[w]_n}^{(n)} ) \right\}.
\end{equation*}
First, note that two complete words  $u,w \in \cA_n^*$ are $\sim_n$ equivalent if and only if $\lvert u \rvert = \lvert w \rvert$. Indeed, for any word $v\in \cA_n^*$, $\sum_{i=1}^{|v|-1} h^{(n)}_{v_i} = (|v|-1) h^{(n)} $, where $h^{(n)}$ is the height of every tower. 

In addition, if $w \not \in \cL^{(n)} (\boldsymbol \sigma) $, then $\mu(\sT^{(n)}_{w})=0$. Therefore, we can consider only the complete words of the language of $X_{\boldsymbol \sigma}^{(n)}$. Finally, for $w \in \cC \cL^{(n)} (\boldsymbol \sigma)$ 

\begin{align*}
\mu(\sT^{(n)}_w) 
&= \frac{\mu(\sT^{(n)}_w) }{ \sum_{a \in \cA_n} \mu (\sT^{(n)}_a) } \\
&= \frac{h^{(n)} \mu(B_w^{(n)}) }{ h^{(n)} \sum_{a \in \cA_n} \mu (B_a^{(n)}) } \\
&= \frac{\mu(B_w^{(n)}) }{ \mu (B^{(n)}) } =  \mu^{(n)}  ([w]_{X^{(n)}_{\boldsymbol \sigma}}).
\end{align*}

Thus, we have proved that $\displaystyle \mu ( \sT_{[w]_n}^{(n)} ) = \sum_{\substack{w \in \cC\cL^{(n)}(\boldsymbol \sigma) \\ \lvert w \rvert  = \ell }}  \mu^{(n)}  ([w]_{X^{(n)}_{\boldsymbol \sigma}})$ and we are done. 
\end{proof}

\begin{corollary} \label{cor induced Toeplitz}
Let $\boldsymbol \sigma = (\sigma_n \colon \cA_{n+1}^* \to \cA_n^*)_{n \in \N}$ be a recognizable, constant length and primitive directive sequence. Let $\mu$ be an $S$-invariant ergodic measure on $X_{\boldsymbol \sigma}$ and $ \mu^{(m)} $ the induced measure on $X^{(m)}_{\boldsymbol{\sigma}}$ for $m \in \N$. Then
\begin{equation} \label{equality deltamu}
\delta_{\mu} = \delta_{ \mu^{(m)}} 
\end{equation}
\end{corollary}
\begin{proof}
As was proven in \cref{theorem partial rigidity rate}, $\displaystyle \left(\sup_{\ell \geq 2} \left\{ \sum_{\substack{w \in \cC\cL^{(n)}(\boldsymbol \sigma) \\ \lvert w \rvert  = \ell }}  \mu^{(n)}  ([w]_{X^{(n)}_{\boldsymbol \sigma}}) \right\} \right)_{n \in \N}$ is a decreasing sequence. Therefore,
\begin{equation*}
\inf_{n \geq 1}  \sup_{\ell \geq 2} \left\{ \sum_{\substack{w \in \cC\cL^{(n)}(\boldsymbol \sigma) \\ \lvert w \rvert  = \ell }}   \mu^{(n)}  ([w]_{X^{(n)}_{\boldsymbol \sigma}}) \right\} = \inf_{n \geq m} \sup_{\ell \geq 2} \left\{ \sum_{\substack{w \in \cC\cL^{(n)}(\boldsymbol \sigma) \\ \lvert w \rvert  = \ell }}   \mu^{(n)}  ([w]_{X^{(n)}_{\boldsymbol \sigma}}) \right\}. 
\end{equation*}
So, equality \eqref{equality deltamu} follows from \cref{theorem toeplitz delta mu}.
\end{proof}

\begin{remark}
The previous corollary can be useful in the following situation. If we know the partial rigidity rate of $(X_{\boldsymbol \sigma}, \cB, \mu, S)$, then for $\boldsymbol \tau = (\phi_1, \ldots, \phi_r, \sigma_0, \sigma_1, \ldots)$ where $(\phi_1, \ldots, \phi_r)$ is a finite sequence of primitive and recognizable constant length morphisms, the system $(X_{\boldsymbol \tau}, \cB, \nu, S)$ has the same partial rigidity rate as that of $(X_{\boldsymbol \sigma}, \cB, \mu, S)$. 
\end{remark}

As a consequence, we obtain \cref{thm subtitutions intro} announced in the introduction.

\begin{corollary}[\cref{thm subtitutions intro}] \label{teolargocste}
 Let $\sigma\colon \cA^* \rightarrow \cA^*$ be a recognizable, primitive and constant length substitution. Let $\mu$ be the unique $S$-invariant measure on $(X_{\sigma},S)$. Then,
\begin{equation}
\delta_{\mu} = \sup_{\ell \geq 2} \left\{ \sum_{\substack{w \in \cC\cL( \sigma) \\ \lvert w \rvert  = \ell }} \mu (  [w]_{X_{\sigma}} ) \right\}.
\end{equation}
\end{corollary}

\begin{proof}
Note that if $\boldsymbol \sigma$ is the directive sequence made up only of $\sigma$, then $X^{(n)}_{\boldsymbol \sigma} = X_{\sigma}$ for every $n \geq 1$. So, for every $n \geq 1$ and $\ell \geq 2$
\begin{equation*}
\sum_{\substack{w \in \cC\cL^{(n)}(\boldsymbol \sigma) \\ \lvert w \rvert  = \ell }}  \mu^{(n)}  ([w]_{X^{(n)}_{\boldsymbol \sigma}}) = \sum_{\substack{w \in \cC\cL( \sigma) \\ \lvert w \rvert  = \ell }} \mu (  [w]_{X_{\sigma}} ),
\end{equation*}
which implies \eqref{eq Toeplitz delta mu}.
\end{proof}

The last corollary allows us to conclude that the necessary condition for rigidity of \cref{teoLRrigid} is also a sufficient condition in the case of constant length substitutions.  

\begin{corollary} \label{cor rigid constant length} 
Let $\sigma\colon \cA^* \rightarrow \cA^*$ be a primitive and constant length substitution. Let $\mu$ be the unique $S$-invariant measure on $(X_{\sigma},S)$. Then
$(X_{\sigma}, \cB, \mu,S)$ is rigid if and only if
\begin{equation} \label{lim p/q 2}
    \limsup_{m \rightarrow \infty} \frac{q_{X_{ \sigma}}(m)}{p_{X_{ \sigma}}(m)} = 1,
\end{equation}
where  $p_{X_{ \sigma}}(m) = | \cL_m(X_{\sigma})|$ and $q_{X_{ \sigma}}(m)= | \cC \cL_m(X_{\sigma})|$. 
\end{corollary}
\begin{proof}
Remark that by \cite{Mosse1992} the substitution is recognizable. Note that \cref{teoLRrigid} gives one implication. For the other, remark that by \cref{lemacomplex}, if $W_n = \cL_n(\sigma) \backslash \cC\cL_n(\sigma)$, the limit \eqref{lim p/q 2} implies that there is a sequence $(m_n)_{n \in \N}$ such that $\displaystyle\lim_{n \to \infty} \sum_{w \in W_{m_n}}\mu ([w]_{X_{\sigma}})=0$. Thus, 
\begin{equation*}
\lim_{n \to \infty} \sum_{w \in \cC\cL_{m_n}(\sigma)}\mu ([w]_{X_{\sigma}})= 1,
\end{equation*}
and by \cref{teolargocste} $(X_{\sigma}, \cB, \mu, S)$ is rigid. 
\end{proof}

\subsection{The Thue-Morse family} \label{sec:Thue-Morse family}

In this section, we compute the partial rigidity rate for a family of constant length substition subshifts. This family of substitutions is inspired by the Thue-Morse substitution. The measures of the cylinder sets of the substition subshifts are computed using techniques from \cite{Queffelec:1987}. 

We need to fix some additional notions and notation. For a substitution $\sigma \colon \cA^* \rightarrow \cA^*$,  $i,j \in \N$ and a nonempty word $v \in  \cA^*$ we define $\sigma_{i,j} (v)$ as the word obtained from $\sigma(v)$ by deleting its first $i$ and last $j$ letters (assuming $i+j<|\sigma(v)|$). We say that a nonempty word $w \in \cA ^*$ admits an interpretation $s= (v_1v_2 \ldots v_m, i,j)$ if $\sigma_{i, j}(v_1 \ldots v_m) = w$, where $i<|\sigma(v_1)|$ and $j < |\sigma(v_m)|$. We denote $a(s) = v_1 \ldots v_m$ and say that $a(s)$ is an ancestor of $w$. The set of interpretations of $w$ is called $I(w)$. We will need the following result.

\begin{lemma}{{\cite[Theorem 3]{Frid_frequency_D0L:1998}}} \label{lemma medidas cilindros ancestros}
Let $\sigma\colon \cA^* \rightarrow \cA^*$ be a primitive substitution and $\lambda$ be the Perron-Frobenius eigenvalue of $M_\sigma$. If $\mu$ is the unique $S$-invariant measure on $X_{\sigma}$, then
\begin{equation*}
\mu([w]_{X_{\sigma}}) = \frac{1}{\lambda} \sum_{s \in I(w) } \mu ([ a(s)]_{X_{\sigma}}).
\end{equation*}
\end{lemma}

Notice that for a primitive and constant length substitution $\sigma\colon \cA^* \to \cA^*$, the Perron-Frobenius eigenvalue of $M_\sigma$ is $\lVert \sigma \rVert$.

From now on, we assume that the alphabet $\cA$ is a finite abelian group with addition $+$. Let $u = u_1 u_2 \ldots u_{\ell}$ be a nonempty word in $\cA^*$. We define the morphism $\sigma_{u} \colon \cA^* \rightarrow \cA^*$ as follows:
\begin{equation*}
\sigma_{u}(g) = (u_1 + g) (u_2 + g) \ldots (u_{\ell} + g) \qquad \forall g \in \cA.
\end{equation*}
For example, if we consider $\cA = \{0,1\}$ with the addition modulo $2$, and $u = 01$, then $\sigma_u$ is the Thue-Morse substitution.

We call $\sigma_{u}$ the \emph{Thue-Morse type substitution given by} $u$. It is easy to check that for any finite abelian group $\cA$, and any $u\in \cA^{\ast}$, the morphism $\sigma_u$ is recognizable. The family of Thue-Morse type substitutions is closed under composition, meaning that if $u, v \in \cA^*$, with $|u| = \ell$ and $|v|= r$, then
\begin{align*}
\sigma_v \circ \sigma_u(g)  &=  \sigma_v ((u_1 + g) (u_2 + g) \ldots (u_{\ell} + g) ) \\ 
&= (v_1 + u_1 + g) \ldots (v_r + u_1 + g) \ldots (v_1 + u_{\ell} + g) \ldots (v_r + u_{\ell} + g) \\
&= \sigma_w(g),
\end{align*}
 where $w= \sigma_v(u)$. 
 
Note that $X_{\sigma_u}$ is not always infinite (for example, if $u=aa$ for some $a\in \cA$). Even if we require $u$ to contain all the letters of $\cA$ we may have a finite $X_{\sigma_u}$ (for example, for $u=010$, the substitution $0 \mapsto 010$ and $1 \mapsto 101$ gives rise to a periodic subshift). 

In all the following, we will assume that $X_{\sigma_u}$ is infinite. The following lemma shows a recursive formula that will be very useful for computing the partial rigidity rate of many Thue-Morse type substition subshifts. 

\begin{lemma} \label{lemmarecurrencia}
Let $(X_{\sigma_u}, \cB, \mu, S)$ be an infinite substition subshift, where $\sigma_u\colon \cA^* \to \cA^*$ is a primitive Thue-Morse type substitution and $\lvert u \rvert = \ell$. Then, for every $n \geq 1$ and $i \in \{1, \ldots, \ell\}$
{\small
\begin{equation*}
\mu( C_{n \ell + i} (+ g)) = \frac{1}{\ell} \left( \sum_{k=1}^{\ell +1-i} \mu( C_{n + 1} (- u_{k+i-1} + u_k + g)) + \sum_{k=1}^{i-1} \mu( C_{n + 2} (- u_{k} + u_{k-i+1} + g))  \right),
\end{equation*}
}
where
$$ C_{m} (+ g) = \bigcup_{\substack{w = w_1 \ldots w_m \in \cL(\sigma_{u}) \\ \textrm{ such that } w_m = w_1 + g}} [w]_{X_{\sigma_u}} .$$
\end{lemma}

\begin{proof}
We define $\cW_m (+ g) = \{w = w_1 \ldots w_m \in \cL(\sigma_{u}) : w_m = w_1 + g \}$ (so that $C_{m} (+ g) = \bigcup_{w \in \cW_{m} (+ g)} [w]_{X_{\sigma_u}}$).

Let us fix $g \in \cA$, $n\geq 1$ and $i \in \{1,\ldots,\ell\}$. Let $w \in \cW_{\ell n + i} (+ g)$ and assume that $w$ has a unique interpretation. Thus, it has a unique ancestor $v=a(w)$ (this last assumption does not play a role other than simplifying the notation; see \cref{obs ancestros}). Note that $v$, being an ancestor of $w$, can have length $n+1$ or $n+2$.
    
If $|v|=n+1$ and $(\sigma_{u})_{j,j'} = w$, then $ |\sigma_{u}(v_1)| - j + |\sigma(v_{n+1})| - j' + |\sigma_{u}(v_2 \ldots v_n)| =  (n+1) \ell - j - j' = |w| = n\ell + i $. Therefore, $\ell - i = j + j'$  and $j \in \{0,\ldots,\ell-i\}$. Then, if $k = j+1$, it follows that $\ell - j' = k + i -1$. Putting together the following three equalities:
\begin{align*}
w_1 &= \sigma_{u}(v_1)_k = u_k + v_1 \\
w_{n \ell + i} &= \sigma_{u}(v_{n+1})_{k+i-1} = u_{k+i-1} + v_{n+1} \\
w_{n \ell + i} &= w_1 + g & (  w \in \cW_{\ell n + i} (+ g) )
\end{align*}
we derive that $u_k + v_1 + g = u_{k+i-1} + v_{n+1}$ and it follows that  $ v_{n+1} = (-u_{k+i-1} + u_k + g) + v_1 $ (note that here is the only moment where we use that $(\cA, +)$ is abelian). Then, $v \in \cW_{n+1}(-u_{k+i-1} + u_k + g) $. 

A similar reasoning is used for the case $|v|=n+2$, in which it follows that $u_{k-i+1} + v_1 + g = u_{k} + v_{n+1}$ and thus $v \in \cW_{n+2}(-u_{k} + u_{k-i+1} + g) $ for $k \in \{1,\ldots,i-1\}$. 

Finally, it is clear that every word $v \in \cW_{n+1}(- u_{k+i-1} + u_{k-i+1} + g) \cup \cW_{n+2}(- u_{k} + u_{k-i+1} + g)$ is an ancestor of a word $w \in \cW_{\ell n + i} (+ g)$. Hence, from \cref{lemma medidas cilindros ancestros} it follows that
\begin{align*}
&\mu( C_{n \ell + i} (+ g)) = \sum_{w \in \cW_{\ell n + i} (+ g)} \mu([w]_{X_{\sigma_u}}) =  \sum_{w \in \cW_{\ell n + i} (+ g)} \frac{1}{\ell} \mu([a(w)]_{X_{\sigma_u}}) \\
&= \frac{1}{\ell} \left( \sum_{v \in  \cW_{n+1}(- u_{k+i-1} + u_k  + g)} \mu([v]_{X_{\sigma_u}}) + \sum_{v \in  \cW_{n+2}(- u_{k} + u_{k-i+1} + g)  } \mu([v]_{X_{\sigma_u}}) \right) \\
&= \frac{1}{\ell}  \left( \sum_{k=1}^{\ell +1-i} \mu( C_{n + 1} (- u_{k+i-1} + u_k + g)) + \sum_{k=1}^{i-1} \mu( C_{n + 2} (- u_{k} + u_{k-i+1} + g))  \right). \qedhere
\end{align*}
\end{proof}

\begin{remark} \label{obs ancestros}
In the previous proof, it was artificially assumed that if $w \in \cW_{\ell n + i} (+ g)$, then it has a single ancestor. However, it is possible to make the same proof without that assumption. If it has more than one interpretation, they would appear in the last summation of the proof. Strictly speaking, the last summation of the proof corresponds exactly to the following expression:
\begin{equation*}
\frac{1}{\ell}  \sum_{w \in \cW_{\ell n + i} (+ g) } \left( \sum_{s \in I(w) } \mu ([ a(s)]_{X_{\sigma_u}}) \right),
\end{equation*}
which is nothing more than the measure of $C_{\ell n + i} (+ g)$ (see \cref{lemma medidas cilindros ancestros}). 
\end{remark}

\begin{remark}
If $0 \in \cA$ is the neutral element of the group $(\cA, +)$, then
\begin{equation*}
\mu( C_{n \ell + i} (+ 0)) = \frac{1}{\ell} \left( \sum_{k=1}^{\ell +1-i} \mu( C_{n + 1} (- u_{k+i-1} + u_k )) + \sum_{k=1}^i \mu( C_{n + 2} ( -u_{k} + u_{k-i+1}))  \right)
\end{equation*}
and for every $g \in \cA$ it follows that
\begin{equation*}
\mu( C_{n \ell + 1} (+ g)) = \mu( C_{n +1 } (+ g)).
\end{equation*}
\end{remark}

\begin{lemma} \label{lemma menor que 1}
Under the same assumptions of the previous lemma, for all $m \geq 2$ and $g \in \cA$, it follows that
\begin{equation*}
\mu( C_{m} (+ g) ) < 1.
\end{equation*}
\end{lemma}

\begin{proof} 
If $\mu( C_{m} (+ g) ) = 1$, then every word of length $m$ satisfies that $w_m = w_1 + g$, so every word would be uniquely extensible. Indeed, for $a \in \cA$, if $w_1 w_2 \ldots w_m a \in \cL(\sigma_{u})$, then $w_2 \ldots w_m a \in \cL_m (\sigma_{u}) $ and therefore $a = w_2 + g$. This would imply that $p(m) = p(m+1)$, and thus that the system is periodic, which is a contradiction. 
\end{proof}

With both lemmas, we can prove that Thue-Morse type substition subshifts are not rigid. 

\begin{proposition} \label{teo cota explicita}
Let $(\cA, +)$ be a finite abelian group and let $u \in \cA^*$ be a nonempty word. Assume that $\sigma_{u}$ is a primitive substitution and that $X_{\sigma_{u}}$ is infinite. Then, the system $(X_{\sigma_{u}}, \cB, \mu, S)$ is not rigid, where $\mu$ is the unique invariant measure. Moreover, 
\begin{equation} \label{eq maxcste}
\delta_{\mu} \leq  \max_{\substack{i \in \{2,\ldots, \ell\} \\ g \in \cA } } \mu( C_{ i} (+ g)).
\end{equation}
\end{proposition}

\begin{proof}
Let us first note that if $0 \in \cA$ is the neutral element of the group $(\cA, +)$, then by \cref{teolargocste} and using the notation of \cref{lemmarecurrencia}, we have that
\begin{equation*}
\delta_{\mu} = \sup_{m > 1} \mu ( C_m(+ 0)).
\end{equation*}

Then, by induction, it follows that for all $m \geq 2$, $ \mu ( C_m(+ 0))$ is a convex combination of elements in the set $\{\mu ( C_i(+ g) ) \in [0,1) : 2 \leq i \leq \ell, g \in \cA \}  $. Therefore, for every $m \geq 2$ it follows that $  \mu ( C_m(+ 0)) \leq \max \{\mu ( C_i(+ g) ) \in [0,1) : 2 \leq i \leq \ell, g \in \cA \} $ and we conclude inequality \eqref{eq maxcste} with \cref{teolargocste}. 
    
Using \cref{lemma menor que 1}, the maximum is strictly smaller than $1$, so the system is not rigid. 
\end{proof}

Now we can compute the partial rigidity rate for the Thue-Morse substitution. 

\begin{theorem} \label{thm:partial_rig_ThueMorse}
\label{cor thue morse}
The partial rigidity rate of the Thue-Morse subshift is $\delta_{\mu} =\frac{2}{3}$. 
\end{theorem}

\begin{proof}
First note that $\mu(C_2(+0)) = \mu([00]_{X_{\sigma}}) + \mu([11]_{X_{\sigma}}) = \frac{1}{3}$ and $\mu(C_2(+1)) = \mu([01]_{X_{\sigma}}) $ $ + \mu([10]_{X_{\sigma}}) = \frac{2}{3}$. So, $\delta_{\mu} \leq \max \{ \frac{1}{3}, \frac{2}{3} \} = \frac{2}{3}$. 

Second, $\mu(C_4(+0)) = \mu([0010]_{X_{\sigma}}) + \mu([0100]_{X_{\sigma}}) + \mu([0110]_{X_{\sigma}}) + \mu([1011]_{X_{\sigma}}) + \mu([1101]_{X_{\sigma}}) + \mu([1001]_{X_{\sigma}}) = \frac{2}{3}$. So, $\delta_{\mu} \geq \frac{2}{3}$ and we are done. 
\end{proof}

\begin{remark} \label{remark seq Thue-Morse}
Once we know that the partial rigidity rate of the Thue-Morse substitution subshift is $2/3$, we can state that the $\delta_{\mu}$-partial rigidity sequence is $(3 \cdot 2^n)_{n \in \N}$. Indeed, if we take $w(n) \equiv 0100$, it is straightforward to check that $\mu(\sT^{(n)}_{[0100]_{n}}) = 2/3$ and by \cref{remark partial seq} the partial rigidity sequence is $m_n = h^{(n)}_0 + h^{(n)}_1 + h^{(n)}_0 = 3 \cdot 2^n$. 
    
This fact is interesting because we know that the ``natural'' rigidity sequence of the maximal equicontinuous factor of this system (the odometer) is $(2^n)_{n \in \N}$ and with the same proof we can only guarantee that $(2^n)_{n \in \N}$ is a $\frac{1}{3}$-partial rigidity sequence of the Thue-Morse substitution subshift. In any case, $(3 \cdot 2^n)_{n \in \N}$ is also a rigidity sequence for the odometer.
\end{remark}

\begin{example} The same type of reasoning can be done for $\cA = \{0,1,2\}$ (with addition mod 3), $u= 0100$, so that 
$\sigma_{u} (0)= 0100$, $\sigma_{u} (1)= 1211$, $\sigma_{u} (2)=2022$.  Here, we find that $\displaystyle  \max_{\substack{i \in \{2,3,4\} \\ g \in \cA } } \mu( C_{ i} (+ g)) = \frac{1}{2}$ and $\mu(C_4(+0)) = \frac{1}{2}$, so $\delta_{\mu} = \frac{1}{2}$.

Note that in particular this system and the Thue-Morse subshift are not measurably isomorphic, because their partial rigidity rates are different.
\end{example}

Thanks to Thue-Morse type substitutions, we can prove that if $\Delta_s$ is the set of partial rigidity rates for substitution subshifts, then the number $1$ is an accumulation point of $\Delta_s$. More precisely:

\begin{corollary} \label{cor muy rigido}
For every $j \geq 1$ and $d \geq 2$, the primitive substitution $\sigma_j: \cA^* \rightarrow \cA^*$ with $\cA = \{0,1, \ldots, d-1\}$ given by
\begin{align*}
\sigma_j(0) &= 0^j 1^j \\
\sigma_j(1) &= 1^j 2^j \\
&\vdots \\
\sigma_j(d-1) &= (d-j)^j 0^j,
\end{align*}
satisfy that $1 - \frac{1}{j} \leq \delta_{\mu} < 1 $, where $\mu$ is the unique invariant measure of $(X_{\sigma_j},S)$.
\end{corollary}

\begin{proof}
Set $j \leq 1$, $d \leq 2$ and let $(X_{\sigma_j}, \cB, \mu, S)$ be the substitution subshift defined by $\sigma_j$. Then, by \cref{prop m consecutive Toeplitz}, $ \frac{j-1}{j} \leq \delta_{\mu} $. In addition, the substitution $\sigma_j$ clearly corresponds to a Thue-Morse type substitution in $\cA = \Z / d \Z $ given by the word $0^j1^j$, so by \cref{teo cota explicita}, $\delta_{\mu} < 1$.  
\end{proof}

\begin{question} \label{question multiple partial rigidity}
    
Notice that in this section all systems are uniquely ergodic so there is no ambiguity in computing their partial rigidity rates. However, partial rigidity is a purely measure-theoretic concept: a non-uniquely ergodic topological system could have two ergodic measures with two different partial rigidity rates. For example, there are Toeplitz subshifts with positive topological entropy and that have measures for which they are measurably isomorphic to odometers (see, for instance, \cite{Williams_toeplitz_non_ue:1984}). For such a system, there is a measure $\mu_1$ with positive entropy, so that it is not rigid and a measure $\mu_2$ that is rigid. 

When this paper was first released on arXiv, we suggested that it would be interesting to exhibit examples of zero entropy or finite alphabet rank minimal $\cS$-adic subshifts with multiple partial rigidity rates. In particular construct systems such that the set $\{\delta_{\mu} : \mu \text{ is an ergodic measure} \}$ is as large as the set of their ergodic measures. This was fully answered by the third author for the finite alphabet rank case in \cite{radic2024multiple} and it is still open for systems with infinitely many ergodic measures.
\end{question}

\subsection{Every number is a partial rigidity rate} \label{sec:every number}
It is natural to wonder what values in $[0,1]$ can be partial rigidity rates of a system. We establish the following:
\begin{theorem} \label{thm arbitrary partial rigidity rate}
For every $\delta \in [0,1]$ there is a measure preserving system $(X, \cX, \mu, T)$ such that $\delta = \delta_{\mu}$. 
\end{theorem}

 \begin{remark} \label{remark Friedman}
    
 In his seminal paper \cite{Friedman_partial_mixing_rigidity_factors:1989}, Friedman proved that for every $\delta \in (0,1)$ there is an ergodic system $(X, \cX, \mu, T)$ and a sequence $(n_k)_{k \in \N}$ (depending on $\delta$)  such that 
\begin{equation} \label{eq p mixing and p rigid}
    \lim_{k \to \infty} \mu(A \cap T^{-n_k}B) = \delta \mu(A \cap B) + (1-\delta) \mu(A) \mu(B) \quad \forall A,B \in \cX.
\end{equation}
It is clear that \eqref{eq p mixing and p rigid} implies partial rigidity with respect to the sequence $(n_k)_{k \in \N}$ and that $\delta$ is the largest possible constant associated with that sequence. It is unclear that in Friedman's construction $(n_k)_{k\in \N}$ is the sequence that maximizes the partial rigidity rate (recall that in the definition of partial rigidity rate, all sequences are considered). Indeed, a system can exhibit very different behaviors along different sequences. For instance, a system can be 1-rigid for one sequence and 1-mixing for another (consider, for example, a weakly mixing and rigid transformation).

Our approach uses a completely different construction than Friedman's. This allows us to guarantee that $\delta$ is the actual partial rigidity rate. However, we note that, unlike Friedman's construction, our systems are not ergodic.
\end{remark}

To prove  \cref{thm arbitrary partial rigidity rate}, we will introduce a family of substitutions similar to the family used in the proof of \cref{cor muy rigido}.
For $\ell \geq 2$, let $\zeta_{\ell}\colon \{0,1\}^* \to \{0,1\}^*$ be the substitution given by $\zeta_{\ell}(0) = 01^{\ell - 1}$,
$\zeta_{\ell}(1)= 1 0^{\ell - 1}$. 

Notice that for every $\ell \geq 2$, $\zeta_{\ell}$ corresponds to a Thue-Morse type substitution in $\cA = \Z / 2 \Z $ given by the word $01^{\ell-1}$.

\begin{proposition} \label{prop very rigid family}
The partial rigidity rate for $(X_{\zeta_{\ell}}, \cB, \mu, S)$ is $\frac{\ell-1}{\ell +1}$ for every $\ell \geq 6$ and the partial rigidity sequence is $(\ell^k)_{k \in \N}$.
\end{proposition}

\begin{proof}
First, using classical methods for computing the measure of cylinder sets for constant length substitutions from \cite{Queffelec:1987}, we have that $\mu([00]_{X_{\zeta_{\ell}}}) = \mu([11]_{X_{\zeta_{\ell}}}) = \frac{\ell -1}{2(\ell +1)}$ and $\mu([01]_{X_{\zeta_{\ell}}})=\mu([10]_{X_{\zeta_{\ell}}}) = \frac{1}{\ell + 1}$. Therefore, $\mu([00]_{X_{\zeta_{\ell}}} \cup [11]_{X_{\zeta_{\ell}}}) = \frac{\ell-1}{\ell +1}$ and so, by \cref{teolargocste}, $\frac{\ell-1}{\ell +1} \leq \delta_{\mu}$.

Using \cref{teo cota explicita}, it suffices to show that
\begin{equation} \label{eq maxmax}
\max\left\{ \sum_{w \in \cC\cL(\zeta_{\ell}) \cap \cA^{i}} \mu([w]_{X_{\zeta_{\ell}}}), \sum_{w \in \cA^{i} \backslash \cC\cL(\zeta_{\ell}) } \mu([w]_{X_{\zeta_{\ell}}})
\right\} \leq \frac{\ell-1}{\ell +1} \quad \quad \forall \ i \in\{2,\ldots, \ell\},
\end{equation}

 In order to prove \eqref{eq maxmax}, it is not hard to see that for $3 \leq i \leq \ell$,
\begin{equation*}
\cL_{i}(\zeta_{\ell}) \backslash \cC\cL(\zeta_{\ell}) = \bigcup_{j=1}^{i -1 } (\{ 0^j 1^{i -j} \} \cup \{ 1^j 0^{i -j} \} ). 
\end{equation*}
Also, as for the Thue-Morse substitution, $\mu([w]_{X_{\zeta_{\ell}}}) = \mu([\overline{w}]_{X_{\zeta_{\ell}}})$, where $\overline{0} = 1$ and $\overline{1} =0$, so we can focus the analysis only on the words of the form $0^j 1^{i -j}$. Thus, after computing $\zeta_{\ell} (00) = 01^{\ell-1} 01^{\ell-1}$, $\zeta_{\ell} (01) = 01^{\ell-1} 10^{\ell-1}$, $\zeta_{\ell} (10) = 10^{\ell-1} 01^{\ell-1}$ and $\zeta_{\ell} (11) = 10^{\ell-1} 10^{\ell-1}$, it is clear that $01^{i-1}$  has only one interpretation in $0$. Therefore, by \cref{lemma medidas cilindros ancestros},
\begin{equation*}
\mu([01^{i-1}]_{X_{\zeta_{\ell}}}) = \frac{1}{\ell} \mu([0]_{X_{\zeta_{\ell}}}) = \frac{1}{2 \ell}.
\end{equation*}

For $0^{i-1}1$ there are two interpretations, on the word $11$ and on the word $10$. So using \cref{lemma medidas cilindros ancestros} again we get, 
\begin{equation*}
\mu([0^{i-1}1]_{X_{\zeta_{\ell}}}) = \frac{1}{\ell} \left (\mu([10]_{X_{\zeta_{\ell}}}) + \mu([11]_{X_{\zeta_{\ell}}})\right)=\frac{1}{ 2 \ell }.
\end{equation*}
    
In the other cases, $2 \leq j \leq i - 2$, $0^j 1^{i-j}$ has only one interpretation on the word $10$, so \cref{lemma medidas cilindros ancestros} gives,
\begin{equation*}
\mu([0^j1^{i-j}]_{X_{\zeta_{\ell}}}) = \frac{1}{\ell} \mu([10]_{X_{\zeta_{\ell}}}) = \frac{1}{ \ell (\ell +1)}.
\end{equation*}

Finally, for every $3 \leq i \leq \ell$, with $\ell \geq 6$ by assumption, we have:
\begin{equation*}
\sum_{w \in \cL_{i}(\zeta_{\ell}) \backslash \cC\cL(\zeta_{\ell}) } \mu([w]_{X_{\zeta_{\ell}}}) = 2 \left(\frac{1}{2 \ell} + \frac{1}{2 \ell} + \sum_{j=2}^{i-2} \frac{1}{\ell (\ell+1)}\right) = \frac{1}{\ell} \left( 2 + 2 \frac{i-3}{\ell +1 } \right) \leq \frac{1}{2}.
\end{equation*}
This implies that for every $2 \leq i \leq \ell$ the maximum in equation \eqref{eq maxmax} is achieved for the sum over complete words of length $i$. Also, for every $3 \leq i \leq \ell$,  
\begin{equation*}
\sum_{\substack{w \in  \cC\cL(\zeta_{\ell}) \\ |w|=i} }\mu([w]_{X_{\zeta_{\ell}}}) = 1 - \frac{1}{\ell} \left( 2 + 2 \frac{i-3}{\ell +1 }\right) \leq \frac{\ell -1}{\ell +1}.
\end{equation*}

Therefore, the maximum in \eqref{eq maxcste} is achieved for complete words of length $2$. So, by \cref{teo cota explicita} and \cref{teolargocste}, $\delta_{\mu} = \frac{\ell-1}{\ell +1}$. Reasoning as in \cref{remark seq Thue-Morse}, we can deduce that the partial rigidity sequence for $\delta_{\mu}$ is equal to $(h^{(k)})_{k \in \N} = (\ell^k)_{k \in \N}$. 
\end{proof}

Using the family of substitutions studied in \cref{prop very rigid family}, for any $\delta \in [0,1]$, we can construct a system whose partial rigidity rate is equal to $\delta$. In the following proof, we will work with different parameters $\ell$, but in order not to overload the notation, the unique invariant measure of $X_{\zeta_{\ell}}$ will be simply denoted by $\mu$.

\begin{proof}[Proof of \cref{thm arbitrary partial rigidity rate}]

The property is clear for $\delta=0$ (e.g., taking a mixing system) and $\delta=1$ (taking a rigid system). Let $\delta\in (0,1)$. Notice that for every $\varepsilon_0 >0$ there exist $\ell \geq 6$ such that for some $m_0 \geq 1$, $(\frac{\ell - 1}{\ell +1 })^{m_0+1}  \leq \delta \leq (\frac{\ell - 1}{\ell +1 })^{m_0}  $, where $1 - \frac{\ell - 1}{\ell +1 } \leq \varepsilon_0$ (and so, $(\frac{\ell - 1}{\ell +1 })^{m_0}  - (\frac{\ell - 1}{\ell +1 })^{m_0 + 1} \leq (\frac{\ell - 1}{\ell +1 })^{m_0} \varepsilon_0 \leq \varepsilon_0 $).  

By \cref{prop:prop_parcial_rigid} and \cref{prop very rigid family} $(\frac{\ell - 1}{\ell +1 })^{m_0}$ is the partial rigidity rate of $(X_{\zeta_{\ell}}^{m_0}, \cB(X_{\zeta_{\ell}}^{m_0}), \mu^{m_0}, S \times \cdots \times S)$, so $\delta$ can be $\varepsilon_0$-approximated by a partial rigidity rate. 
Recall that for $(X_{\zeta_{\ell}}^{m_0}, \cB(X_{\zeta_{\ell}}^{m_0}), \mu^{m_0}, S \times \cdots \times S)$, the partial rigidity sequence is still $(\ell^k)_{k \in \N}$. If $\delta = (\frac{\ell -1}{\ell +1})^{m_0}$, then there is nothing more to prove. If not, we will construct inductively a better approximation, and, in each step, we will implicitly assume that the equality is not achieved, since otherwise we would have completed the construction.

Set $\varepsilon_1 =\varepsilon_0/2$. We can find an integer $\ell'$ such that $(\frac{\ell - 1}{\ell +1 })^{m_0} (\frac{\ell' - 1}{\ell' +1 })^{m_1+1}   \leq \delta \leq (\frac{\ell - 1}{\ell +1 })^{m_0}(\frac{\ell' - 1}{\ell' +1 })^{m_1}$, with $(\frac{\ell' - 1}{\ell' +1 }) \leq \varepsilon_1$. 
Without loss of generality, we may assume that $\ell'$ is a power of $\ell$ (i.e., $\ell' = \ell^{q_1}$), and
therefore $(\frac{\ell - 1}{\ell +1 })^{m_0}(\frac{\ell' - 1}{\ell' +1 })^{m_1} $ is the partial rigidity rate of $(X_{\zeta_{\ell}}^{m_0} \times X_{\zeta_{\ell'}}^{m_1}  , \cB(X_{\zeta_{\ell}}^{m_0} \times X_{\zeta_{\ell'}}^{m_1} ), \mu^{m_0} \otimes \mu^{m_1}, S \times \cdots \times S)$, where the partial rigidity sequence can be taken as  $(\ell, \ell^{q_1}, \ell^{2q_1}, ...) $, which is a subsequence of $(\ell^k)_{k \in \N}$ and equal to $((\ell^{q_1})^{k})_{k \in \N}$ except for the first term.

Then, inductively, for $\varepsilon_k=\epsilon_0/2^k$ we can find sequences $(m_k)_{k \in \N}$, $(q_k)_{k \in \N}$ and $(\ell_k)_{k \in \N}$ such that $\ell_0 = \ell$, $\ell_{k+1} = \ell^{q_k}_{k}$, and denoting $\delta_k=\prod_{i=0}^{k
} \left(\frac{l_i -1}{l_i+1}\right)^{m_i}$, we have 

\begin{equation*}
\left(\frac{l_k -1}{l_k+1}\right) 
 \delta_k  \leq \delta \leq  \delta_k, 
\end{equation*}
where $1 - \varepsilon_k \leq \frac{\ell_k - 1}{\ell_k +1} < 1$. Therefore, $\delta_{k} \searrow \delta$ as $k$ goes to $\infty$. 

By construction, $\delta_k$ is the partial rigidity rate of $$\displaystyle \mathscr{X}_k = \left(X_k, \cB(X_k), \mu^{m_0}  \otimes \cdots \otimes \mu^{m_k}, S \times \cdots \times S\right),$$ where $\displaystyle X_k = \prod_{i = 0}^{k} X_{\zeta_{\ell_i}}^{m_i}$.

Let $\mathscr{X}_{\infty}$ be the inverse limit of $(\mathscr{X}_k)_{k \in \N}$, where the factor map $\pi_k \colon \mathscr{X}_{k+1} \to \mathscr{X}_k$ is the canonical projection. Thus, by \cref{prop:prop_parcial_rigid}, the partial rigidity rate of $\mathscr{X}_{\infty}$ is precisely $\displaystyle \delta = \inf_{k \in \N} \delta_k$.
\end{proof}

\begin{question} \label{question arbitrary delta}
    \cref{thm arbitrary partial rigidity rate} states that any number $\delta \in [0,1]$ can be the partial rigidity rate of a measure preserving system $(X, \cX, \mu,T)$. We pose the question of whether this holds true when restricted to a particular class of systems, for instance, ergodic, finite rank, Toeplitz, or weakly mixing. In the same direction, concerning the set of partial rigidity rates for substitution subshifts $\Delta_s$, we ask: is it dense in $[0,1]$? Note that $\Delta_s $ is countable, because there are only countable many substitutions (up to conjugacy), so density in [0,1] is the most we can expect.
\end{question}

\section{Open questions}
 This section gathers open questions that arose in the paper for the reader's convenience, and some further discussions.

\textbf{Partial mixing:} As mentioned above, when partial rigidity was introduced in \cite{Friedman_partial_mixing_rigidity_factors:1989}, it was closely related to the notion of partial mixing. It would be interesting to give a characterisation of this property for $\cS$-adic subshift. In particular, we do not know whether any of the examples presented in this work are partially mixing or not.

\textbf{Finite topological rank mixing subshifts (\cref{question finite rank}):}
Give a characterization or condition on a directive sequence $\boldsymbol \sigma$ so that it defines an $\cS$-adic subshift that admits either a mixing or a partially mixing measure. 
Find a uniquely ergodic finite alphabet rank $\cS$-adic subshift such that it is not mixing nor partially rigid.

\textbf{Combinatorial characterization of rigidity:}  Can \cref{teoLRrigid} be extended to a broader class of $\cS$-adic subshifts?
For instance, the limit \eqref{lim q/p} also holds for arbitrary Sturmian systems which are all rigid, but not necessarily linearly recurrent subshifts (see \cite[Section 3.4]{Radic_tesis:2023}). We also prove in  \cref{cor rigid constant length} that the limit \eqref{lim q/p} is also a sufficient condition for rigidity in the case of constant length substitution subshifts. Is it sufficient also for other classes of subshifts?

\textbf{System with multiple partial rigidity rates (\cref{question multiple partial rigidity}):} Find systems of zero entropy such that the set $\{\delta_{\mu} : \mu \text{ is an ergodic measure} \}$ is as large as the set of their ergodic measures.  

  \textbf{Realizing partial rigidity rates (\cref{question arbitrary delta}):}   Prove \cref{thm arbitrary partial rigidity rate} for more restrictive classes systems. Similarly, determine if the set of partial rigidity rates for substitution subshifts $\Delta_s$ is dense in $[0,1]$.

\end{document}